\documentclass[12pt]{amsart}
\usepackage{amsfonts,amsthm,amssymb,mathrsfs,setspace,pstricks,booktabs,xcolor,mathtools,amsmath,geometry,graphicx,subcaption}
\usepackage{tikz}
\usetikzlibrary{arrows.meta, positioning}
\usepackage{scalerel}
\usepackage{hyperref}

\usepackage{stmaryrd}
\hypersetup{
	colorlinks=true,
	linkcolor=blue,
	citecolor=red,
	filecolor=red,      
	urlcolor=magenta,
}

\theoremstyle{plain}
\newtheorem{theorem}{Theorem}[section]
\newtheorem{lemma}[theorem]{Lemma}
\newtheorem{corollary}[theorem]{Corollary}
\newtheorem{proposition}[theorem]{Proposition}

\theoremstyle{definition}
\newtheorem{definition}[theorem]{Definition}

\theoremstyle{remark}
\newtheorem{remark}[theorem]{Remark}

\title{Decompositions of ZMC Graphs and Euler-Ramanujan type identities}
    
\author{Priyank Vasu}
\address{Department of Mathematics, Indian Institute of Technology Patna, Bihta, Patna--801106, Bihar, India}
\email{priyank\_2121ma16@iitp.ac.in}

\author{Sam K. Mathew}
\address{International Centre for Theoretical Sciences, Bengaluru, India}
\email{sam.mathew@icts.res.in, samkmathewresearch@gmail.com}

\author{Rahul Kumar Singh}
\address{Department of Mathematics, Indian Institute of Technology Patna,
Bihta, Patna--801106, Bihar, India}
\email{rahulks@iitp.ac.in}

\author{Rukmini Dey}
\address{International Centre for Theoretical Sciences, Bengaluru, India}
\email{rukmini@icts.res.in}


\begin{document}
\maketitle

\begin{abstract}
In this paper, we study finite and infinite decomposition formulas for zero mean curvature (ZMC) graphs in Euclidean, Lorentz--Minkowski, and isotropic (3)-spaces. We first derive new Euler--Ramanujan-type identities that decompose the conjugate of Scherk's first minimal surface into dilated catenoids. Using Weierstrass factorisation and power series methods, we then obtain infinite decompositions for a broad class of isotropic ZMC graphs into helicoids, logarithmoids of revolution, and Enneper surfaces. These results are extended to wider families of ZMC surfaces arising from the López--Ros transformation, Bonnet rotation, and a one-parameter family of metric deformations. We also establish finite decomposition formulas, including analogues of Scherk tower decompositions in Euclidean and isotropic settings, and prove a characterisation theorem for finite decompositions of isotropic minimal surfaces. Finally, we discuss applications to lamellar structures.
\end{abstract}

	\section{Introduction}
It is now a well established fact that a connection exists between the height functions of certain zero mean curvature surfaces in the Euclidean space $\mathbb E^3:= (\mathbb{R}^3,dx^2+dy^2+dz^2)$ and the Lorentz-Minkowski space $\mathbb{E}_1^3 := (\mathbb{R}^3,dx^2+dy^2-dz^2)$ and some specific Euler-Ramanujan identities. The first such connection appeared in a paper by the physicist R.~Kamien \cite{Kamien2001}, where a particular Euler-Ramanujan identity was used to show that Scherk’s first minimal surface in $\mathbb E^3$ can be decomposed into an infinite superposition of dilated helicoids.
	 
	Further connections between Euler--Ramanujan identities and zero mean curvature surfaces were developed by Dey and collaborators \cite{Dey2016, DeySingh2017, DeySarmaSingh2020}, who used Weierstrass--Enneper representations and Born--Infeld solitons to derive new complex identities. Lone \cite{Lone2021} extended this viewpoint to affine minimal translation surfaces. Kamien \cite{Kamien2001} showed that Scherk's first minimal surface admits finite decompositions into dilated copies of itself, an idea later developed by Dey et al.\ \cite{DeyGhoshSoundararajan2023} to obtain decomposition formulas for Scherk's first and second minimal surfaces and related minimal and maximal surfaces. More recently, Paul et al.\ \cite{Paul2025} combined Euler--Ramanujan identities with Wick rotations to represent a one-parameter family of Scherk-type zero mean curvature surfaces in $\mathbb{E}_1^3$ as infinite superpositions of dilated helicoids, along with several finite decomposition formulas.

Despite these developments, existing decomposition results remain largely restricted to specific examples, as their constructions rely heavily on particular Euler--Ramanujan identities and are therefore often case-specific. In this work, we instead develop a geometric framework independent of any particular Euler--Ramanujan identity. This approach identifies broader classes of zero mean curvature surfaces admitting decompositions and provides a unified perspective on previously known examples.

To motivate our first result, recall that every minimal surface has an associated conjugate minimal surface, obtained by replacing the coordinate functions in a conformal parametrisation by their harmonic conjugates. This leads to the following question.

\begin{center}
\parbox{0.92\textwidth}{
\textbf{Question 1.}
Does the conjugate of Scherk's first minimal surface admit a Kamien-type decomposition? More precisely, can its height function be expressed as an infinite superposition of suitably dilated copies of the conjugate helicoid?
}
\end{center}

Our first objective is to answer Question~1. In Section~2, we establish such a decomposition relating Scherk's tower to the catenoid.

Having established a decomposition for a particular surface, it is natural to ask whether such representations occur more generally among zero mean curvature surfaces. This leads to the following question.

\begin{center}
\parbox{0.92\textwidth}{
\textbf{Question 2.}
Which classes of zero mean curvature surfaces admit representations as infinite superpositions of suitably dilated surfaces? Can one characterize families for which Kamien-type decompositions hold?
}
\end{center}

Question~2 is the central theme of the remainder of the paper. In Section~3, we study it in isotropic \(3\)-space and identify a family of isotropic minimal surfaces admitting such decompositions. We then extend the framework to Euclidean and Lorentz--Minkowski spaces, obtaining analogous results for minimal and maximal surfaces. This shows that the phenomenon belongs to a broader geometric framework rather than isolated examples.

Motivated by Kamien's finite decomposition \cite{Kamien2001}, we ask whether analogous finite representations occur more broadly among zero mean curvature surfaces.

\begin{center}
\parbox{0.92\textwidth}{
\textbf{Question 3.}
Which classes of zero mean curvature surfaces admit finite decompositions into suitably dilated copies of themselves?
}
\end{center}
In Section~5, we investigate this problem and identify broad classes of zero mean curvature surfaces admitting finite decomposition formulas.

The paper is organised as follows. In Section~2, we motivate a general theory of decompositions for graphical ZMC surfaces through the model example of Scherk's tower. We derive several identities expressing Scherk's tower as an infinite superposition of catenoids in both Euclidean and Lorentz--Minkowski spaces. These identities naturally lead to isotropic \(3\)-space, where the decomposition takes a particularly simple form: the isotropic Scherk tower decomposes into logarithmoids of revolution, which are isotropic minimal surfaces. Since graphical isotropic minimal surfaces are harmonic functions, this setting provides a natural framework for studying such decompositions. The results of this section also motivate the use of Weierstrass product factorisation in the subsequent sections and a deeper study of graphical ZMC surfaces in isotropic \(3\)-space.

	In Section~3, we establish several decomposition results for graphical ZMC surfaces in isotropic \(3\)-space. We first consider height functions of the form \(\operatorname{Re}(\ln f(x+iy))\) or \(\operatorname{Im}(\ln f(x+iy))\), a broad class that includes the examples from Section~2. Using the Weierstrass product theorem, we show that dilated logarithmoids of revolution and helicoids serve as building blocks for these two classes of isotropic minimal surfaces. We then prove, via power series expansions of holomorphic functions, that every isotropic minimal surface admits a decomposition into generalized Enneper surfaces. Finally, we study Bonnet rotations and López--Ros deformations, obtaining decompositions involving both dilated helicoids and logarithmoids of revolution, weighted by \(\cos\theta\) and \(\sin\theta\), respectively.

	In Section~4, we deform the metric of a graphical isotropic minimal surface by a real parameter \(c\) and study the resulting decomposition. We obtain decompositions analogous to those above, but the building blocks are no longer merely dilated or rotated helicoids, logarithmoids of revolution, or generalized Enneper surfaces. Instead, they are composed with an additional function whose geometric interpretation need not be as simple as a dilation or rotation.

	In Section~5, we turn to finite decomposition theory. We first study finite decompositions of Scherk's tower into dilated copies of itself in the Euclidean and isotropic settings, and then establish a general finite decomposition theorem for isotropic minimal surfaces. In Section~6, we apply this theory to the study of twisted grain boundaries.

	\section{New Examples of Decompositions of ZMC Graphs and Euler-Ramanujan type Identities}
	As noted in the introduction, Scherk's doubly periodic minimal surface admits a decomposition into helicoids via an Euler--Ramanujan identity. Since the conjugates of Scherk's first surface and the helicoid are Scherk's tower and the catenoid, respectively, we investigate the corresponding decomposition of Scherk's tower into catenoids.

    For completeness, we recall some standard definitions. A surface \(S\subset \mathbb{E}^3\), locally represented as the graph of a smooth function \(f(x,y)\), is minimal if \(f\) satisfies
\begin{equation}\label{eqn: MSE}
(1+f_x^2)f_{yy}-2f_xf_yf_{xy}+(1+f_y^2)f_{xx}=0.
\end{equation}
Similarly, a graphical zero mean curvature surface in Lorentz--Minkowski space \(\mathbb{E}_1^3\), with height function \(g(x,y)\) over a domain in the spacelike \(xy\)-plane, satisfies the maximal surface equation
\begin{equation}\label{eqn: MaxSE}
(1-g_x^2)g_{yy}+2g_xg_yg_{xy}+(1-g_y^2)g_{xx}=0.
\end{equation}

	\subsection{Decomposition in Euclidean Space \(\mathbb{E}^3\)}
	
	In this subsection, we investigate the decomposition relation between Scherk's tower and the catenoid in three-dimensional Euclidean space \(\mathbb{E}^3\). We begin with the infinite product representation of the sine function:
	\begin{equation}\label{eqn: product-sin}
		\sin a = a \prod_{k \neq 0} \left( \frac{k\pi + a}{k\pi} \right).
	\end{equation}
	Substituting \(a = x + iy\) and taking the natural logarithm of both sides yields:
	\begin{equation}
		\ln\left[\sin(x+iy)\right] = \ln(x+iy) + \sum_{k \neq 0} \ln\left[ \frac{k\pi + x + iy}{k\pi} \right].
	\end{equation}
	For notational convenience, let \(x^k = x + \pi k\). Using the standard identity \(\ln(x+iy) = \ln\left[\sqrt{x^2 + y^2}\right] + i\tan^{-1}\left(\frac{y}{x}\right)\), we obtain:
	
	\noindent
	\(\sum_{k \neq 0} \ln\left[ \frac{x^k + iy}{\pi k} \right] + \ln(x+iy) = \sum_{k \neq 0} \ln\left[ \sqrt{ \frac{(x^k)^2 + y^2}{\pi^2 k^2} } \right] + \ln\left[ \sqrt{x^2 + y^2} \right] + i \sum_{k=-\infty}^{\infty} \tan^{-1}\left( \frac{y}{x^k} \right).\)
	Consequently, we arrive at:
	\begin{equation}
		\label{eqn: StSsHC base equation}
		\ln\left[\sin(x+iy)\right] = \sum_{k \neq 0} \ln\left[ \sqrt{ \frac{(x^k)^2 + y^2}{\pi^2 k^2} } \right] + \ln\left[ \sqrt{x^2 + y^2} \right] + i \sum_{k=-\infty}^{\infty} \tan^{-1}\left( \frac{y}{x^k} \right).
	\end{equation}
	The imaginary part of the left-hand side is \(\tan^{-1}\left[ \tanh(y) \cot(x) \right]\). Equating this with the imaginary part of the right-hand side yields the celebrated Euler-Ramanujan identity:
	\begin{equation}
		\tan^{-1}\left[ \tanh \alpha \cot \beta \right] = \sum_{k=-\infty}^{\infty} \tan^{-1}\left( \frac{\alpha}{\beta + k\pi} \right),
	\end{equation}
	which, as demonstrated in previous works, provides an infinite decomposition of Scherk's first surface in terms of helicoids.
	
	We now consider the real part of equation \eqref{eqn: StSsHC base equation}. First, we compute the real part of \(\ln\left[\sin(x+iy)\right]\). Using the identity \(\sin(x+iy)=\sin(x)\cosh(y)+i\cos(x)\sinh(y)\), we have: 
	\(\Re\left[\ln\left[\sin(x+iy)\right]\right]=\ln\left[\sqrt{\sin^2(x)\cosh^2(y)+\cos^2(x)\sinh^2(y)}\right].\) The expression under the square root can be simplified by employing \(\cosh^2(y)=1+\sinh^2(y)\) and \(\cos^2(x)=1-\sin^2(x)\): \(\Re\left[\ln\left[\sin(x+iy)\right]\right]=\ln\left[ \sqrt{\sin^2(x)+\sinh^2(y)} \right].\) Equating this with the real part of the right-hand side of equation \ref{eqn: StSsHC base equation} yields:
	\[
	\sum_{k\neq0}\ln\left[\sqrt{\frac{(x^k)^2+y^2}{\pi^2 k^2}}\right] +\ln\left[\sqrt{x^2+y^2}\right]=\ln\left[ \sqrt{\sin^2(x)+\sinh^2(y)} \right].
	\]
	Exponentiating both sides and squaring eliminates the logarithm and square root, giving:
	\begin{equation}
		\label{eqn: ER type ScC identity}
		(x^2+y^2)\cdot\prod_{k\neq 0}\left[\frac{(x+k\pi)^2+y^2}{\pi^2 k^2}\right]=\sin^2(x)+\sinh^2(y).
	\end{equation}
    This we name as the \textbf{Conjugate Euler-Ramanujan Identity}. 
	Setting \(y=0\) in equation \ref{eqn: ER type ScC identity} yields the familiar product representation for \(\sin^2(x)\):
	\begin{equation}
		\label{eqn: ER type ScC identity intial}
		x^2\cdot\prod_{k\neq 0}\left[\frac{(x+k\pi)^2}{\pi^2 k^2}\right]=\sin^2(x).
	\end{equation}
	Having established these analytic identities, we now turn to their geometric interpretation. Consider the implicit surface defined by: 
	\[
	\cosh^2(az)-\sinh^2(ay)-2\sin^2(bx)-c=0.
	\]
	We determine the parameters \(a\), \(b\), and \(c\) for which this equation represents Scherk's tower. Using the identities \(\cosh^2(A)=1+\sinh^2(A)\) and \(\sin^2(A)=\frac{1-\cos(2A)}{2}\), together with the hyperbolic identity \(\sinh^2(A)-\sinh^2(B)=\sinh(A+B)\sinh(A-B)\) (see \cite{JD}, section 2.5.1.3, equation 8), we rewrite the surface as: \(\sinh(a(z+y))\sinh(a(z-y))+\cos(2bx)-c=0.\) Choosing \(c=0\) and \(b=1/2\) simplifies this to: \(\cos(x)=\sinh(a(y+z))\sinh(a(y-z)).\)
	
	Observe that the equation: \(\cos(x)=\sinh\left( \frac{y+z}{\sqrt{2}} \right)\sinh\left(\frac{y-z}{\sqrt{2}}\right)\) corresponds to Scherk's tower rotated by \(45^\circ\) in the \(yz\)-plane. Thus, we choose \(a=\frac{1}{\sqrt{2}}\), which yields the implicit equation for Scherk's tower:
	\[
	\cosh^2\left(\tfrac{z}{\sqrt{2}}\right)-\sinh^2\left(\tfrac{y}{\sqrt{2}}\right)-2\sin^2\left(\tfrac{x}{2}\right)=0.
	\]
	Solving for \(z\) gives the explicit graphical representation:
	\begin{equation}
		\label{eqn: graph scherk tower}
		f_{\text{ScherkTower}}=\sqrt{2}\cosh^{-1}\Big[\sqrt{\sinh^2\left(\tfrac{y}{\sqrt{2}}\right)+2\sin^2\left(\tfrac{x}{2}\right)}\,\Big].
	\end{equation}
	Consequently, Scherk's tower must satisfy the relation: \(\cosh^2\left(\frac{f_{\text{ScherkTower}}}{\sqrt{2}}\right)=\sinh^2\left(\frac{y}{\sqrt{2}}\right)+2\sin^2\left(\frac{x}{2}\right).\) Applying equations \eqref{eqn: ER type ScC identity} and \eqref{eqn: ER type ScC identity intial} to this relation, we obtain:
	\begin{equation}
		\label{eqn: ER type ScC identity Modified}
		\Big[\left(\tfrac{x}{2}\right)^2+\left(\tfrac{y}{\sqrt{2}}\right)^2\Big]\cdot\prod_{k\neq 0}\Bigg[\frac{\left(\tfrac{x}{2}+k\pi\right)^2+\left(\tfrac{y}{\sqrt{2}}\right)^2}{\pi^2 k^2}\Bigg]=\sinh^2\left(\tfrac{y}{\sqrt{2}}\right)+\sin^2\left(\tfrac{x}{2}\right).
	\end{equation}
	We now define a family of surfaces given implicitly by:
	\begin{equation}
		\label{eqn: Dilated Catenoids Implicit}
		\cosh^2(z_k)=\left(\tfrac{x}{2}+k\pi\right)^2+\left(\tfrac{y}{\sqrt{2}}\right)^2\quad\text{and}\quad \cosh^2(z_0)=\left(\tfrac{x}{2}\right)^2+\left(\tfrac{y}{\sqrt{2}}\right)^2.
	\end{equation}
	These correspond precisely to dilated catenoids. Their explicit graphical representations are given by:
	\begin{equation}
		\label{eqn: Dilated Catenoids Graph}
		f^k_{\text{Dil.Cat}}=\cosh^{-1}\Big[\sqrt{\left(\tfrac{x}{2}+k\pi\right)^2+(\tfrac{y}{\sqrt{2}})^2}\;\Big]
		\quad\text{and}\quad
		f^0_{\text{Dil.Cat}}=\cosh^{-1}\Big[\sqrt{\left(\tfrac{x}{2}\right)^2+(\tfrac{y}{\sqrt{2}})^2}\Big].
	\end{equation}
	From equations \eqref{eqn: Dilated Catenoids Graph} and \eqref{eqn: ER type ScC identity Modified}, we deduce:
	  \footnotesize 
    \[
      \cosh^2\left(f^0_{\text{Dil.Cat}}\right)\prod_{k\neq 0}\frac{\cosh^2\left(f^k_{\text{Dil.Cat}}\right)}{\pi^2k^2}=\sinh^2\left(\tfrac{y}{\sqrt{2}}\right)+\sin^2\left(\tfrac{x}{2}\right)    
    \]
     and 
    \[
      \cosh^2\left(f^0_{\text{Dil.Cat}}\right)\prod_{k\neq 0}\frac{\cosh^2\left(f^k_{\text{Dil.Cat}}\right)}{\pi^2k^2}\Bigg|_{(x,0)}=\sin^2\left(\tfrac{x}{2}\right).
    \]
    \normalsize 
    Combining this result with the graphical representation of Scherk's tower from equation \eqref{eqn: graph scherk tower} yields the following. 
	
	\begin{proposition}
	    \label{thm: Catenoid=Scherk's tower id}
		\textbf{(Catenoid-Scherk's Tower Identity)}
        The identity
        \begin{align*}
            \cosh^2\left(f^0_{\text{Dil.Cat}}\right)\cdot \prod_{k\neq 0}\frac{\cosh^2\left(f^k_{\text{Dil.Cat}}\right)}{\pi^2k^2}\Bigg|_{(x,y)}&+&\left(\cosh^2\left(f^0_{\text{Dil.Cat}}\right)\cdot \prod_{k\neq 0}\frac{\cosh^2\left(f^k_{\text{Dil.Cat}}\right)}{\pi^2k^2}\right)\Bigg|_{(x,0)}\\
            &=&\cosh^2\left(\frac{f_{\text{ScherkTower}}}{\sqrt{2}}\right)\Bigg|_{(x,y)}
        \end{align*}
		is defined on each strip-shaped domains
		\[
		   -2k\pi+2<x<-2(k-1)\pi-2 \quad \text{for each } k \in \mathbb{Z}.
		\]
	\end{proposition}
	To prove this theorem, it is enough to determine the domain on which the identity is well defined. The first term on the left-hand side is defined on the green region in Figure~\ref{pic: Domain Analysis Catenoid Scherk's Tower Identity}. The second is defined on the subregion where the first term is also well defined at \((x,0)\) for every \((x,y)\), namely the blue region. The orange region represents the domain of the right-hand side. Since the blue region is the intersection of these three domains, it is the appropriate domain for the Catenoid--Scherk Tower identity.

We now determine the blue regions explicitly. Their endpoints occur where the boundaries of the holes in the domain of \(f^k_{\text{Dil.Cat}}\) meet the \(x\)-axis. These satisfy

$$
(\tfrac{x}{2}+k\pi)^2=1,
$$

and hence \(x=-2k\pi\pm2\). The corresponding strip boundaries are \(x=-2k\pi+2\) and \(x=-2(k-1)\pi-2\). Thus, the blue regions are the strips

$$
-2k\pi+2<x<-2(k-1)\pi-2,
\qquad k\in\mathbb Z.
$$

	\begin{figure}[htbp]
		\centering
		\begin{subfigure}{0.30\textwidth}
			\centering
			\includegraphics[width=\linewidth]{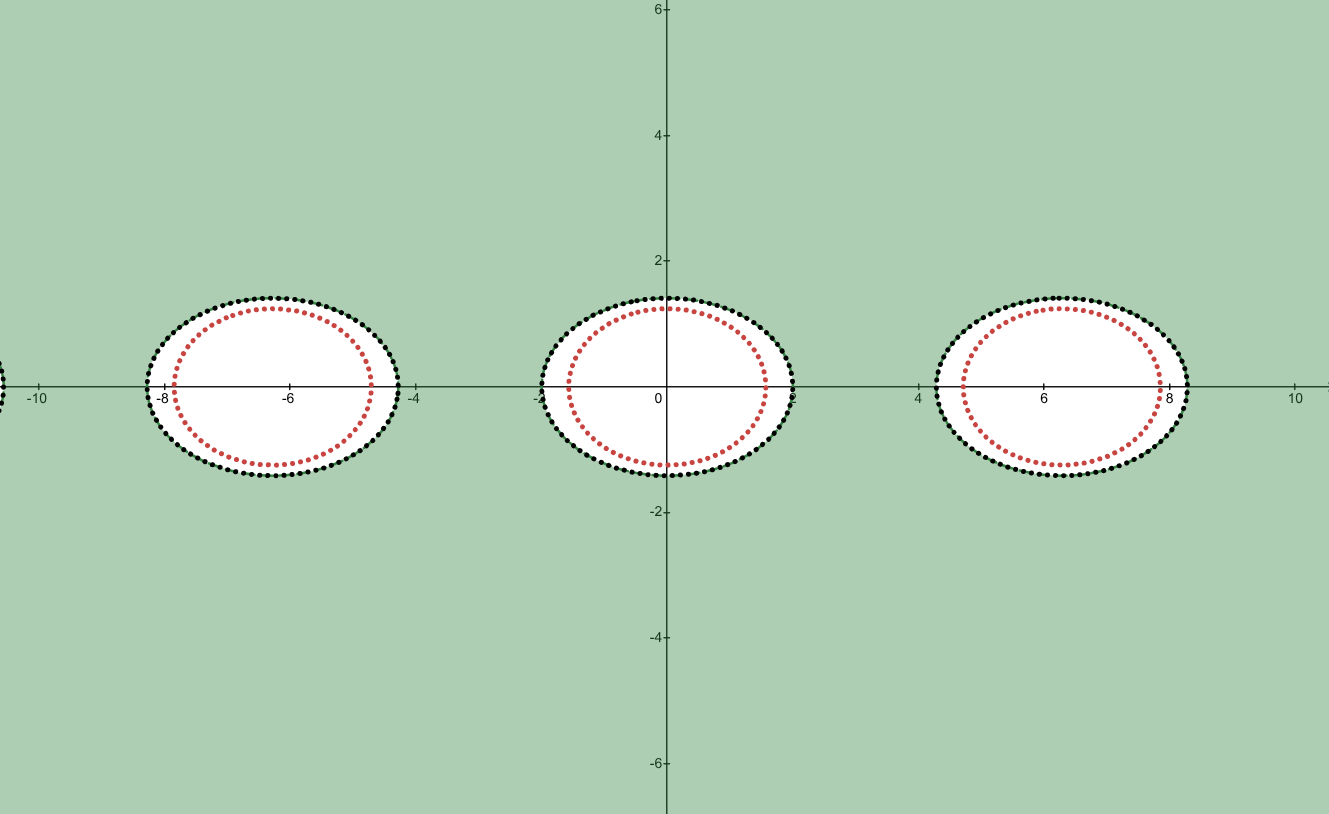}
			\caption{\footnotesize Domain: \(1^{st}\) term of LHS \normalsize}
		\end{subfigure}
		\hfill
		\begin{subfigure}{0.30\textwidth}
			\centering
			\includegraphics[width=\linewidth]{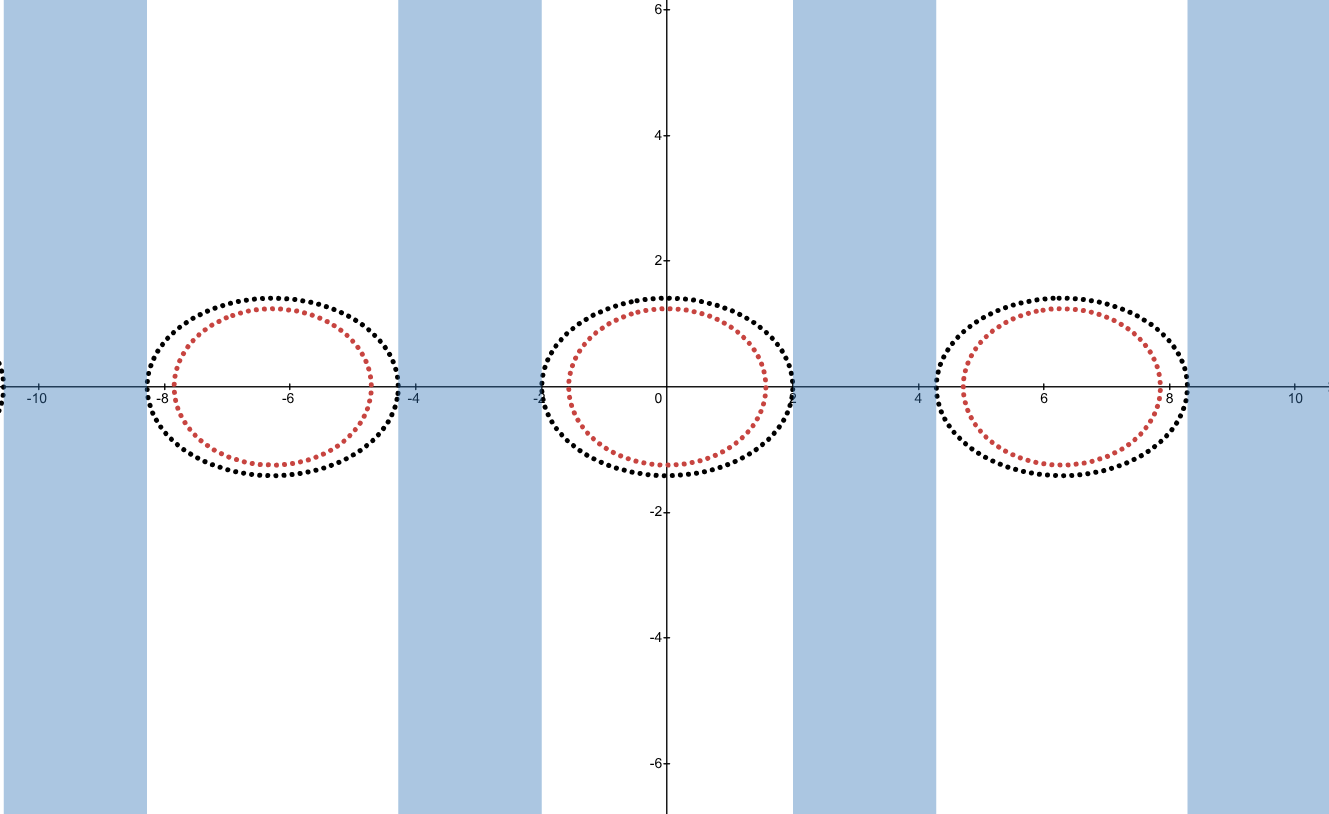}
			\caption{\footnotesize Domain: \(2^{nd}\) term of LHS\normalsize}
		\end{subfigure}
		\hfill
		\begin{subfigure}{0.30\textwidth}
			\centering
			\includegraphics[width=\linewidth]{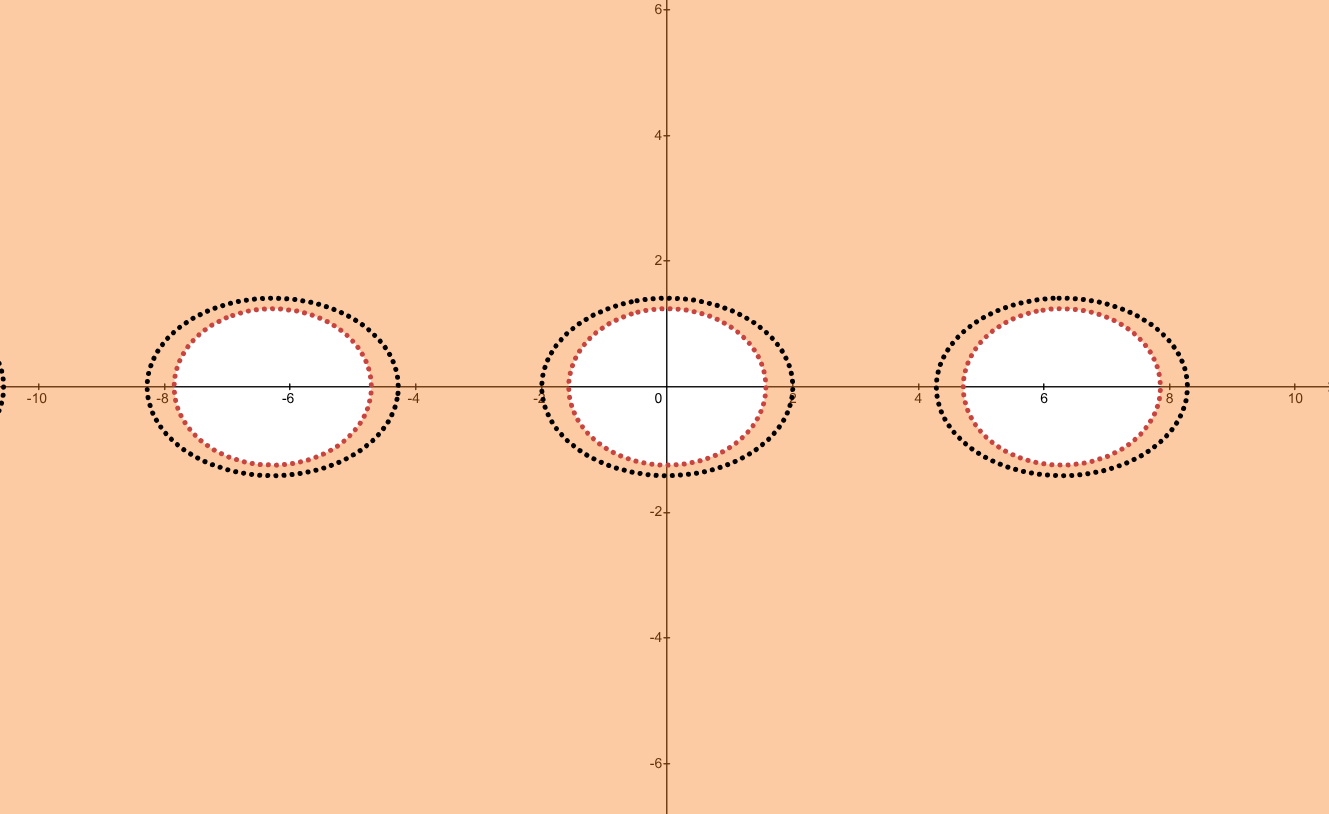}
			\caption{\footnotesize Domain: RHS \normalsize}
		\end{subfigure}
		\caption{\footnotesize Comparison of the domains of the three terms in the Catenoid--Scherk's Tower Identity (Proposition~\ref{thm: Catenoid=Scherk's tower id}). The red ellipses indicate the boundary of the domain of Scherk's Tower, and each black ellipse indicates the boundary of the domain of each dilated catenoid. \normalsize}
		\label{pic: Domain Analysis Catenoid Scherk's Tower Identity}
	\end{figure}
	
	We observe that the second term in the Catenoid--Scherk's Tower identity considerably reduces the size of its domain. We therefore intend to remove the second term from this identity and derive a new identity relating the Catenoid and Scherk's Tower. We begin with the first equation in Equation
	\[
	   \cosh^2\left(f^0_{\text{Dil.Cat}}\right)\prod_{k\neq 0}\frac{\cosh^2\left(f^k_{\text{Dil.Cat}}\right)}{\pi^2 k^2}
	= \sinh^2\left(\frac{y}{\sqrt{2}}\right) + \sin^2\left(\frac{x}{2}\right).
	\]
	We also recall that Scherk's Tower is given by
	\[
	\cosh^2\left(\frac{f_{\text{ScherkTower}}}{\sqrt{2}}\right)
	= \sinh^2\left(\frac{y}{\sqrt{2}}\right) + 2\sin^2\left(\frac{x}{2}\right).
	\]
	From these, we obtain
	\[
	\sinh^2\left(\frac{y}{\sqrt{2}}\right) + \sin^2\left(\frac{x}{2}\right)
	= \frac{1}{2}\left[
	\cosh^2\left(\frac{f_{\text{ScherkTower}}}{\sqrt{2}}\right)\Big|_{(x,y)}
	+
	\cosh^2\left(\frac{f_{\text{ScherkTower}}}{\sqrt{2}}\right)\Big|_{(0,y)}
	\right].
	\]
	Using this relation, we shall derive a second Catenoid--Scherk's Tower identity.
	\begin{proposition}\label{thm: Second Catenoid=Scherk's tower id}
		\textbf{(Second Catenoid-Scherk's Tower Identity)} The identity
		\begin{equation}
			\label{eqn: Second Catenoid-Scherk's Tower Identity}
			\cosh^2\left(f^0_{\text{Dil.Cat}}\right) \cdot \prod_{k\neq0}\frac{\cosh^2\left(f^k_{\text{Dil.Cat}}\right)}{\pi^2k^2}\Bigg|_{(x,y)}=\frac{1}{2}\left[
			\cosh^2\left(\frac{f_{\text{ScherkTower}}}{\sqrt{2}}\right)\Big|_{(x,y)}
			+
			\cosh^2\left(\frac{f_{\text{ScherkTower}}}{\sqrt{2}}\right)\Big|_{(0,y)}
			\right]
		\end{equation}
		is defined on the domain
		\[
		  |y|>\sqrt{2}.
		\]
	\end{proposition}
	The domain for the left-hand side is the green shaded region in Figure~\ref{pic: Domain Analysis Catenoid Scherk's Tower Identity}. The domain of the first term on the right-hand side is the orange shaded region in Figure~\ref{pic: Domain Analysis Catenoid Scherk's Tower Identity}, and the domain of the second term on the right-hand side is the region in which $(0,y)$ is defined for every $(x,y)$ in the domain; this is precisely the blue shaded region in Figure~\ref{pic: Domain Analysis Second Catenoid Scherk's Tower Identity} below. A suitable region where all three regions intersect is the domain of this identity, shown as the grey shaded region in Figure~\ref{pic: Domain Analysis Second Catenoid Scherk's Tower Identity} below. It can be easily seen that the grey shaded domain is given by \(|y|>\sqrt{2}.\)
	\begin{figure}[htbp]
		\centering
		\begin{subfigure}{0.45\textwidth}
			\centering
			\includegraphics[width=\linewidth]{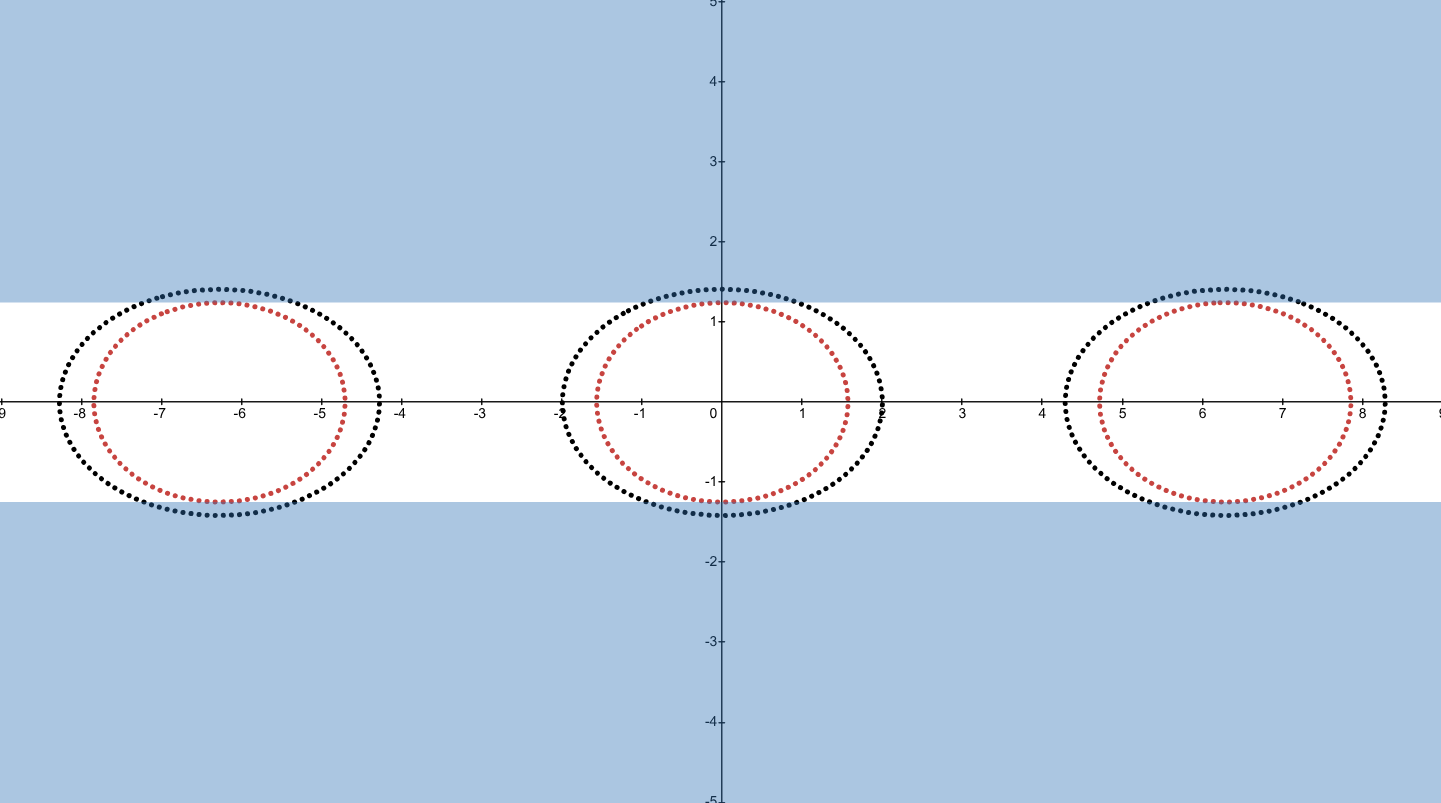}
			\caption{\footnotesize Domain: \(2^{nd}\) term of RHS\normalsize}
		\end{subfigure}
		\hfill
		\begin{subfigure}{0.45\textwidth}
			\centering
			\includegraphics[width=\linewidth]{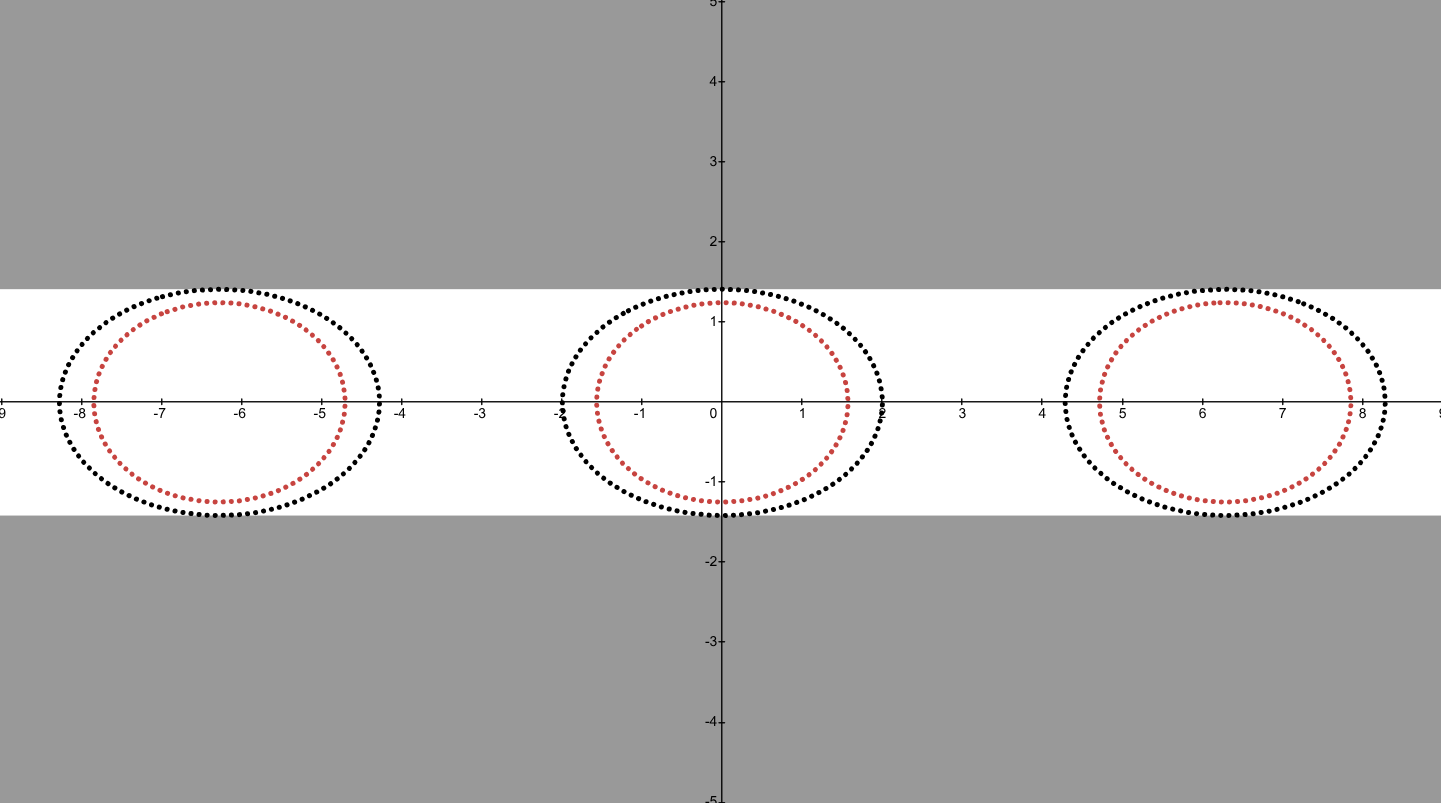}
			\caption{\footnotesize Domain of the identity \normalsize}
		\end{subfigure}
		\caption{\footnotesize Comparison of the domains of the three terms in the Second Catenoid--Scherk's Tower Identity (Proposition~\ref{thm: Second Catenoid=Scherk's tower id}). The red ellipses indicate the boundary of the domain of Scherk's Tower, and each black ellipse indicates the boundary of the domain of each dilated catenoid. \normalsize}
		\label{pic: Domain Analysis Second Catenoid Scherk's Tower Identity}
	\end{figure}
	
	\begin{remark}
		We show in a later section that finite decomposition is also possible.
	\end{remark}

	\subsection{Decomposition in Lorentzian Space \(\mathbb{E}_1^3\) via Wick Rotation}
	The main tool we employ for the decomposition in Lorentz-Minkowski space is the technique of \emph{Wick rotation}. 
	
	\begin{theorem}\label{thm: minimal to maximal}\cite{akamine_wick}
		Let $f(x, y)$ be a solution to \eqref{eqn: MSE} in $\mathbb{E}^3$. If $f$ is even (resp. odd) with respect to the $x$-and $y$-variables, then its Wick rotation $g(x, y)=f(i x, i y)$ (resp. $g(x, y)=-f(i x, i y))$ is a real solution to \eqref{eqn: MaxSE}   in $\mathbb{E}_1^3$ which is spacelike at least near the origin $o=(0,0)$ and conversely.
	\end{theorem}
	\noindent 
	From equation \eqref{eqn: graph scherk tower}, we have \(f_{ScherkTower}=z=\sqrt{2}\cosh^{-1}\big[\sqrt{\sinh^2(\tfrac{y}{\sqrt{2}})+2\sin^2(\tfrac{x}{2})}\,\big]\), which can be rewritten as \(\cosh^2(\tfrac{z}{\sqrt{2}})-\sinh^2(\tfrac{y}{\sqrt{2}})=2\sin^2(\tfrac{x}{2})\). We then interchange the \(x\) and \(z\) variables; this rotation corresponds to a relabeling of coordinates and preserves the zero mean curvature property because the minimal surface equation is invariant under permutations of the coordinate axes when the surface is expressed as a graph over different coordinate planes (see \cite{akamine_wick} for more details). This yields \(\cosh^2(\tfrac{x}{\sqrt{2}})-\sinh^2(\tfrac{y}{\sqrt{2}})=2\sin^2(\tfrac{z}{2})\), which can be expressed graphically as \(z=2\sin^{-1}\frac{1}{\sqrt{2}}\sqrt{\cosh^2(\tfrac{x}{\sqrt{2}})-\sinh^2(\tfrac{y}{\sqrt{2}})}\). The resulting function is even in both \(x\) and \(y\), satisfying the parity condition required for Wick rotation. Applying the Wick rotation to this yields a spacelike solution to \eqref{eqn: MaxSE} corresponding to Scherk's tower: \(z=2\sin^{-1}\frac{1}{\sqrt{2}}\sqrt{\cosh^2(\tfrac{i x}{\sqrt{2}})-\sinh^2(\tfrac{i y}{\sqrt{2}})}\). Using the identities \(\cosh(ix)=\cos(x)\) and \(\sinh(iy)=i\sin(y)\), we obtain the explicit form of what we shall henceforth call the \emph{Spacelike Scherk Tower}:
	\begin{equation}\label{eqn: Sp.ScT}
		f_{\text{Spacelike Scherk Tower}}=2\sin^{-1}\sqrt{
			\tfrac{1}{2}
			\big(\cos^{2}\big(\tfrac{x}{\sqrt{2}}\big)
			+\sin^{2}\big(\tfrac{y}{\sqrt{2}}\big)\big)}.
	\end{equation}
	\noindent
	Implicitly, the spacelike Scherk tower can be expressed as \(2\sin^2\left(\frac{f_{\text{Spacelike Scherk Tower}}}{2}\right)=\cos^{2}\!\left(\frac{x}{\sqrt{2}}\right)+\sin^{2}\!\left(\frac{y}{\sqrt{2}}\right)\). Applying equation \eqref{eqn: ER type ScC identity intial} together with the identity \(\cos^{2}\!\left(\frac{x}{\sqrt{2}}\right)=1-\sin^{2}\!\left(\frac{x}{\sqrt{2}}\right)\), we obtain:
	\begin{equation}\label{eqn: MaxCat-MaxSch.Tow: 1}
		\cos^{2}\!\left(\tfrac{x}{\sqrt{2}}\right)
		+\sin^{2}\!\left(\tfrac{y}{\sqrt{2}}\right)=1-\left(\tfrac{x}{\sqrt{2}}\right)^2\cdot\prod_{k\neq 0}\left[\frac{(\tfrac{x}{\sqrt{2}}+k\pi)^2}{\pi^2 k^2}\right]+\left(\tfrac{y}{\sqrt{2}}\right)^2\cdot\prod_{k\neq 0}\left[\frac{(\tfrac{y}{\sqrt{2}}+k\pi)^2}{\pi^2 k^2}\right]
	\end{equation}
	\noindent
	Recall that the spacelike elliptic catenoid, a maximal surface, is defined by the equation \(z=\sinh^{-1}\left(\sqrt{x^2+y^2}\right)\). Accordingly, we define the following families of dilated elliptic catenoids:
	\[
	f^k_{\text{Dil.S.E.Cat(x)}}=\sinh^{-1}\Bigg[\sqrt{\tfrac{\left(\tfrac{x}{2}+k\pi\right)^2+\left(\tfrac{y}{\sqrt{2}}\right)^2}{\pi^2 k^2}}\;\Bigg]
	\quad\text{and}\quad
	f^0_{\text{Dil.S.E.Cat(x)}}=\sinh^{-1}\Big[\sqrt{\left(\tfrac{x}{2}\right)^2+\left(\tfrac{y}{\sqrt{2}}\right)^2}\;\Big]
	\]
	\[
	f^k_{\text{Dil.S.E.Cat(y)}}=\sinh^{-1}\Bigg[\sqrt{\tfrac{\left(\tfrac{x}{2}\right)^2+\left(\tfrac{y}{\sqrt{2}}+k\pi\right)^2}{\pi^2 k^2}}\;\Bigg]
	\quad\text{and}\quad
	f^0_{\text{Dil.S.E.Cat(y)}}=\sinh^{-1}\Big[\sqrt{\left(\tfrac{x}{2}\right)^2+\left(\tfrac{y}{\sqrt{2}}\right)^2}\;\Big]
	\]
\noindent	
	From the above definitions together with equation \eqref{eqn: MaxCat-MaxSch.Tow: 1}, we deduce:
	\[
	1- \prod_{k=-\infty}^{\infty}\sinh^2\left(f^k_{\text{Dil.S.E.Cat(x)}}\right)\Big|_{(x,0)}+ \prod_{k=-\infty}^{\infty}\sinh^2\left(f^k_{\text{Dil.S.E.Cat(y)}}\right)\Big|_{(0,y)}= \cos^{2}\!\left(\tfrac{x}{\sqrt{2}}\right)+\sin^{2}\!\left(\tfrac{y}{\sqrt{2}}\right).
	\]
\noindent	
	Combining this with equation \eqref{eqn: Sp.ScT} yields the following:

    \begin{proposition}$\textbf{(Spacelike Elliptic Catenoid–Spacelike Scherk Tower Identity)}$ The identity
    \begin{equation}\label{eqn: Sp.Cat-Sp.Sch.Tow Identity}
        1- \prod_{k=-\infty}^{\infty}\sinh^2\left(f^k_{\text{Dil.S.E.Cat(x)}}\right)\Big|_{(x,0)}+\prod_{k=-\infty}^{\infty}\sinh^2\left(f^k_{\text{Dil.S.E.Cat(y)}}\right)\Big|_{(0,y)}
        =2\sin^2\left(\frac{f_{\text{Spacelike ScherkTower}}}{2}\right)\Bigg|_{(x,y)}
    \end{equation}
 is valid over an appropriate domain of \(\mathbb{R}^2\).
    \normalsize
    \end{proposition}
	\begin{remark}
		The decomposition of Scherk's surface into catenoids in both the spaces $\mathbb{E}^3$ and $\mathbb{E}_1^3$ suggests the existence of a space where the corresponding surfaces yield a clean infinite sum decomposition. This is precisely what we will explore in the next subsection, where we discuss the Isotropic space $\mathbb{I}^3$ and its minimal surfaces. In later sections, we will extend this solution to a larger class of minimal surfaces within $\mathbb{I}^3$.
	\end{remark}
	
	\subsection{Decomposition in the Isotropic Space $\mathbb{I}^3$} \label{Decomposition in the Isotropic Space}
	
	The isotropic space $\mathbb{I}^3$ is defined as $\mathbb{R}^3$ equipped with the degenerate metric $ds^2 = dx^2 + dy^2$. We refer to the zero mean curvature (ZMC) surfaces in this space as isotropic minimal surfaces.  We shall use the following Weierstrass--Enneper type representation for isotropic minimal surfaces \cite{daSilva2021,AkamineFujino2022}:
    
	\begin{equation}\label{parabolic normal representation}
	    X_1 = \operatorname{Re} \int^{z} \left( F, -iF, 2FG \right) \, dw 
	\end{equation}
    The pair $(F,G)$ appearing in this representation is called the \emph{Weierstrass data} of the isotropic minimal surface.
    
	Let us consider the specific Weierstrass data $(F, G) = (1, \tfrac{1}{2}\cot z)$. We name isotropic minimal surface given by  Weierstrass--Enneper type representation and its conjugate with this Weierstrass Data to be Isotropic Scherk's Tower and Isotropic Scherk's Surface. Using this same Weierstrass data, we will now determine the corresponding isotropic minimal surfaces.
		\[
	X_1(z,\bar z)=\operatorname{Re} \int^{z} \left( 1, -i, \cot z \right) \, dw
	\quad \text{and} \quad 
	Y_1(z,\bar z)=\operatorname{Im} \int^{z} \left( 1, -i, \cot z \right) \, dw 
	\]
\noindent 	
	This yields the following surfaces:
	\[
	X_1(z) = \left( x, y, \operatorname{Re}[\ln(\sin z)] \right)
	\quad \text{and} \quad  
	Y_1(z) = \left( y, -x, \operatorname{Im}[\ln(\sin z)] \right).
	\]
\noindent	
	Applying the identity from equation \eqref{eqn: StSsHC base equation}, we obtain:
	\begin{equation}\label{eqn: Isotropic StSsHC base equation}
		\operatorname{Re}[\ln(\sin z)] = \sum_{k \neq 0} \ln \sqrt{\tfrac{(x+k\pi)^2 + y^2}{k^2\pi^2}} + \ln \sqrt{x^2 + y^2}
		\quad \text{and} \quad 
		\operatorname{Im}[\ln(\sin z)] = \sum_{k=-\infty}^{\infty} \tan^{-1}\left( \tfrac{y}{x-k\pi} \right)
	\end{equation}
	\noindent
	Consider the Weierstrass data $(F, G) = (1, 1/2z)$. While this data generates the classical catenoid and helicoid in Euclidean space, it yields the \textit{logarithmoid of revolution} and the \textit{isotropic helicoid} within the isotropic setting. The logarithmoid of revolution, defined by the graph $f(x,y) = \ln\sqrt{x^2+y^2}$, serves as the isotropic analogue of the Euclidean catenoid. By defining the dilated and translated logarithmoids as
	\[
	f^{\mathbb{I}^3}_{\mathcal{L}^{(k)}}(x,y) = \ln \sqrt{\tfrac{(x+k\pi)^2 + y^2}{k^2\pi^2}} \quad \text{for } k \neq 0, \quad \text{and} \quad f^{\mathbb{I}^3}_{\mathcal{L}^{(0)}}(x,y) = \ln\sqrt{x^2+y^2}.
	\]
	It follows from equation \eqref{eqn: Isotropic StSsHC base equation} and the definitions of the isotropic Scherk tower and isotropic Scherk's doubly periodic minimal surface that the isotropic Scherk tower admits a decomposition into an infinite sum of these dilated logarithmoids. Analogously, the isotropic Scherk's doubly periodic minimal surface can be expressed as an infinite sum of dilated isotropic helicoids.
	\noindent 
	To summarize, let \(f^{\mathbb{I}^3}_{\mathcal{S}_{\text{tower}}}\) denote the isotropic Scherk tower, \(f^{\mathbb{I}^3}_{\mathcal{S}_{\text{surf}}}\) the isotropic Scherk's doubly periodic minimal surface, \(f^{\mathbb{I}^3}_{\mathcal{L}^{(k)}}\) the dilated logarithmoids of revolution, and \(f^{\mathbb{I}^3}_{\mathcal{H}^{(k)}}\) the dilated helicoids. We arrive at the following  identities:

    \begin{proposition}\label{isotropic Scherk's id}
	\vspace{1em}
	a) \noindent \textbf{The Isotropic Scherk Tower--Logarithmoid of revolution Identity}
	\[
	f^{\mathbb{I}^3}_{\mathcal{S}_{\text{tower}}} = \sum_{k \in \mathbb{Z}} f^{\mathbb{I}^3}_{\mathcal{L}^{(k)}}.
	\]
	
	b)\noindent \textbf{The Isotropic Scherk Surface--Helicoid Identity}
	\[
	f^{\mathbb{I}^3}_{\mathcal{S}_{\text{surf}}} = \sum_{k \in \mathbb{Z}} f^{\mathbb{I}^3}_{\mathcal{H}^{(k)}}.
	\]
    Both identities are valid on any simply connected domain in \(\mathbb{R}^2/\{(k\pi,0);k\in \mathbb{Z}\}\)
	\end{proposition}
\noindent
    The reason we need to set the identity on a simply connected domain is to avoid the complications that may arise from the definition of the branch of the logarithm.

	\section{Decompositions of Isotropic Minimal Surfaces}
	
	In the previous section, we studied decompositions of the Scherk tower and the Scherk doubly periodic minimal surface in various ambient spaces, observing particularly simple infinite decompositions in isotropic \(3\)-space \(\mathbb{I}^3\). In this section, we develop a broader decomposition theory for minimal surfaces in \(\mathbb{I}^3\).
\noindent
A graph \(h(x,y)\) defines an isotropic minimal surface if and only if $
\Delta h=0.
$
Therefore, graphical isotropic minimal surfaces are precisely harmonic functions.
\noindent
We use several methods to obtain such decompositions. In particular, the Weierstrass Product Theorem yields natural decomposition relations among the isotropic analogues of the catenoid, Scherk tower, helicoid, and Scherk doubly periodic surface.

	\subsection{Decomposition via the Weierstrass Product Theorem}

		Since every harmonic function is locally the real or imaginary part of a holomorphic function, graphical isotropic minimal surfaces may be represented accordingly. We focus on height functions of the form
$$
h(x,y)=\operatorname{Re}(\ln f(z))
\quad\text{or}\quad
h(x,y)=\operatorname{Im}(\ln f(z)),
$$
where \(f\) is holomorphic. For the real-part case, we obtain the following infinite decomposition theorem.

		\begin{theorem}\label{thm: weierstrass prod decomposition 1}
			Let $\{a_{n}\}_{n \geq 1}$ be a sequence of nonzero complex numbers, and let $f(z)$ be an entire function with zeros at $\{a_{n}\}$ listed according to their multiplicities. Suppose $f$ has a zero of order $k \geq 0$ at the origin. If $X$ is a graphical isotropic minimal surface with height function $h(x,y) = \operatorname{Re}(\ln f(z))$, then $h$ admits a decomposition on a simply connected domain excluding $\{a_{n}\}_{n \geq 1}$ and $0$ into the following components:
			\begin{enumerate}
				\item A finite sum representing $k$ logarithmoids of revolution, given by $k \ln|z|$;
				\item An infinite sum of dilated and translated logarithmoids of revolution, where the translation and dilation are determined by the factors $(1 - z/a_n)$ associated with the nonzero zeros $a_n$ of $f$;
				\item An infinite sum of graphical isotropic minimal surface given by harmonic polynomials $\operatorname{Re}(P_n(z/a_n))$, derived from the Weierstrass primary factors of the  Weierstrass Product Formula for $f$;
				\item A graphical isotropic minimal surface given by harmonic function defined by $\operatorname{Re}(g(z))$, where $g(z)$ is the entire function appearing in the Weierstrass Product Formula for \(f\).
			\end{enumerate}
			Moreover, since the Weierstrass Product representation of an entire function is not unique, the corresponding decomposition of the isotropic minimal surface is likewise not uniquely determined.
		\end{theorem}
		
		\begin{proof}
			Given that $f$ is an entire function with the specified zeros, the Weierstrass Factorization Theorem (see Apppendix \ref{sec: weierstrass factorisation}, Theorem: \ref{thm: weierstrass factorisation}) guarantees that $f$ can be expressed as: \(f(z) = z^k e^{g(z)} \prod_{n=1}^{\infty} E_n\left( \frac{z}{a_n} \right).\)
			
			Taking the logarithm of both sides and extracting the real part, we obtain:
			\[
			\operatorname{Re}[\ln f(z)] = k \ln|z| + \operatorname{Re}[g(z)] + \sum_{n=1}^{\infty} \ln\left| 1 - \frac{z}{a_n} \right| + \sum_{n=1}^{\infty} \operatorname{Re}\left[ P_n\left( \frac{z}{a_n} \right) \right].
			\]
			This expression provides the precise decomposition for the height function $h(x,y) = \operatorname{Re}(\ln f(z))$, where each term corresponds to the geometric components identified in the proposition.
		\end{proof}
		
		Similarly, for the case where $h$ is defined by the imaginary part of $\ln f(z)$, we obtain the following result regarding its infinite sum decomposition:
		
		\begin{theorem}\label{thm: weierstrass prod decomposition 2}
			Let $\{a_{n}\}_{n \geq 1}$ be a sequence of nonzero complex numbers, and let $f(z)$ be an entire function with zeros at $\{a_{n}\}$ listed according to their multiplicities. Suppose $f$ has a zero of order $k \geq 0$ at the origin. If $X$ is a graphical isotropic minimal surface with height function $h(x,y) = \operatorname{Im}(\ln f(z))$, then $h$ admits a decomposition on a simply connected domain excluding $\{a_{n}\}_{n \geq 1}$ and $0$ into the following components:
			\begin{enumerate}
				\item A finite sum representing $k$ helicoids, given by $k \tan^{-1}(y/x)$;
				\item An infinite sum of dilated and translated helicoids, where each term is given by 
				\[
				\tan^{-1}\left( \frac{a_n^2x - a_n^1y}{|a_n|^2 - a_n^1x - a_n^2y} \right)
				\]
				associated with the nonzero zeros $a_n = a_n^1 + i a_n^2$ of $f$;
				\item An infinite sum of graphical isotropic minimal surface given by harmonic polynomials $\operatorname{Im}(P_n(z/a_n))$, derived from the Weierstrass primary factors in the product formula for $f$;
				\item A graphical isotropic minimal surface given by harmonic function defined by $\operatorname{Im}(g(z))$, where $g(z)$ is the entire function appearing in the Weierstrass Product Theorem.
			\end{enumerate}
			Furthermore, as the Weierstrass product representation of an entire function is not unique, the corresponding decomposition of the isotropic minimal surface is likewise not uniquely determined.
		\end{theorem}
		
		\begin{proof}
			The proof follows immediately by applying the argument of the previous proposition to the imaginary part of the logarithm of the Weierstrass factorization of $f(z)$. Specifically for (2), the expression for \(\tan^{-1}\left( \frac{a_n^2x - a_n^1y}{|a_n|^2 - a_n^1x - a_n^2y} \right)\) can be calculated by expanding the expression for \((1-\tfrac{z}{a_n}):=X+iY\) and calculating \(\tan^{-1}(Y/X)\).
		\end{proof}

		\begin{remark}
			Theorems~\ref{thm: weierstrass prod decomposition 1} and~\ref{thm: weierstrass prod decomposition 2} give decompositions for isotropic minimal surfaces with height functions
$$
h(x,y)=\operatorname{Re}(\ln f(z))
\quad\text{or}\quad
h(x,y)=\operatorname{Im}(\ln f(z)).
$$
These forms are quite general. Indeed, for any holomorphic function \(H(z)\) on a simply connected domain,
$$
H(z)=\ln(e^{H(z)}).
$$
Since isotropic minimal surfaces are precisely the real or imaginary parts of holomorphic functions, taking \(f(z)=e^{H(z)}\) shows that a broad class of such surfaces falls within this framework. Hence, the Weierstrass product decompositions apply quite generally.

		\end{remark}
\noindent		
		Next, we focus our attention on a specific subclass of isotropic minimal surfaces that can be decomposed purely into logarithmoids of revolution or isotropic helicoids. We formalize these cases in the following corollaries:
		
		\begin{corollary}
			Let $h$ be an isotropic minimal surface defined by $h(x,y) = \operatorname{Re}(\ln f(z))$, where $f(z)$ is an entire function admitting a product representation of the form $f(z) = \prod_{n=1}^{\infty}(1 - z/a_n)$. Then $h$ can be decomposed purely into an infinite sum of dilated and translated logarithmoids of revolution on a simply connected domain excluding $\{a_{n}\}_{n \geq 1}$.
		\end{corollary}
		
		\begin{corollary}
			Let $h$ be an isotropic minimal surface defined by $h(x,y) = \operatorname{Im}(\ln f(z))$, where $f(z) = \prod_{n=1}^{\infty}(1 - z/a_n)$. Then $h$ can be decomposed purely into an infinite sum of dilated and translated isotropic helicoids on a simply connected domain excluding $\{a_{n}\}_{n \geq 1}$.
		\end{corollary}
		
		\begin{remark}
			Recall that the isotropic Scherk tower and the isotropic Scherk surface are given by $h(x,y) = \operatorname{Re}(\ln \sin z)$ and $h(x,y) = \operatorname{Im}(\ln \sin z)$, respectively. Since $\sin z$ can be represented purely as a product of linear factors (up to a factor of $z$), they fit the conditions of the above corollaries. This explains why these surfaces admit such clean decompositions into infinite sums of dilated logarithmoids of revolution and dilated isotropic helicoids, respectively.
		\end{remark}

	    Next, we provide a sufficient condition on the function $f$ that guarantees its decomposition purely into isotropic helicoids or logarithmoids of revolution. This can done using the M-test (see Appendix \ref{sec: weierstrass factorisation}, Theorem \ref{thm: M-test}) for the convergence of a product.
		
		\begin{theorem}
			Let $\{a_{n}\}_{n \geq 1}$ be a sequence of nonzero complex numbers such that the sum $\sum_{n=1}^{\infty} |a_n|^{-1} < \infty$. Then the function $f(z) = z^k \prod_{n=1}^{\infty} (1 - z/a_n)$ is a well-defined analytic function on the unit disc $\mathbb{D}$. Consequently, the functions $h$ defined by $h(x,y) = \operatorname{Re}(\ln f(z))$ and $h(x,y) = \operatorname{Im}(\ln f(z))$ define graphical isotropic minimal surfaces on $\mathbb{D}$. These surfaces admit the following decompositions on a simply connected domain excluding $\{a_{n}\}_{n \geq 1}$ and $0$ in $\mathbb{D}$, respectively:
			\begin{enumerate}
				\item A pure infinite sum of dilated and translated logarithmoids of revolution combined with a finite sum of $k$ undilated logarithmoids;
				\item A pure infinite sum of dilated and translated isotropic helicoids combined with a finite sum of $k$ undilated helicoids.
			\end{enumerate}
			In both cases, the dilation and translation of the constituent surfaces are precisely determined by the zeros $\{a_n\}$ of the analytic function $f$.
		\end{theorem}
		
		\begin{proof}
			The analyticity of $f(z)$ on $\mathbb{D}$ is guaranteed by the M-test for products, as the condition $\sum_{n=1}^{\infty} |a_n|^{-1} < \infty$ ensures the uniform convergence of the partial products on compact subsets. This allows the logarithm of $f(z)$ to be expanded into the series $k \ln z + \sum_{n=1}^{\infty} \ln(1 - z/a_n)$, which is well-defined and holomorphic on $\mathbb{D} \setminus \{0, a_n\}$. Taking the real and imaginary parts of this expansion yields harmonic height functions that define isotropic minimal surfaces. These components correspond precisely to the infinite sums of dilated and translated logarithmoids of revolution and isotropic helicoids, respectively.
		\end{proof}
		
		\subsection{Decomposition via Power Series Expansion}
		
		The Weierstrass Product Theorem identifies logarithmoids of revolution and helicoids as natural building blocks of isotropic minimal surfaces. In this subsection, we show that isotropic Enneper surfaces can likewise serve as building blocks, using a power series approach.

Indeed, the height function \(h(x,y)\) of any graphical isotropic minimal surface is harmonic and hence can be represented as the real or imaginary part of an analytic function. Expanding this function into a power series gives
\begin{equation}\label{eqn: power series rep}
h(x,y)=\sum_{n=0}^{\infty}\left(c_n z^n+\bar c_n\bar z^n\right).
\end{equation}

\noindent		
Consider an isotropic graphical minimal surface of the form $\Phi(x, y) = (x, y, h(x, y))$. Because $h$ is real-valued and harmonic, it can be represented as $h(x, y) = f(z) + \overline{f(z)}$ for some holomorphic function $f(z)$, where $z = x + iy$.
\noindent		
For the Weierstrass data \((\omega,g)=(dz,f'(z))\), the Weierstrass--Enneper type representation gives
\begin{align*}
\Phi(z)
&=\Re\int (1,-i,2g),\omega \
&=\Re\int (1,-i,2f'(z)),dz \
&=(x,y,\Re[2f(z)]) \
&=(x,y,h(x,y)).
\end{align*}
\noindent
		This leads to the following lemma:
		
		\begin{lemma}\label{thm: weierstrass data of isotropic minimal surface}
			Let $\Phi(x, y)$ be an arbitrary graphical isotropic minimal surface given by $\Phi(x, y) = (x, y, h(x, y))$. The harmonicity of the height function $h$ implies that $h = f(z) + \overline{f(z)}$ for some holomorphic function $f$. Then, under the Weierstrass data $(dz, f'(z))$, the Weierstrass--Enneper type representation recovers $\Phi(x, y)$ up to a translation.
		\end{lemma}
\noindent		
Applying the lemma to the polynomial-type isotropic minimal surface
\[
c_n z^n+\bar c_n\bar z^n,
\]
we obtain the Weierstrass data
\[
(dz,nc_n z^{n-1})=(dz,a z^{n-1}).
\]
For \(n>1\), this is the Weierstrass data of a scaled generalized Enneper surface of order \(n-1\) in the Euclidean setting. Accordingly, in isotropic \(3\)-space we call the corresponding surface a scaled isotropic generalized Enneper surface of order \(n-1\).

\noindent
For \(n=0\), the surface is a plane parallel to the \(xy\)-plane, while for \(n=1\) it is a plane not necessarily parallel to the \(xy\)-plane. Combined with the expansion
\[
h(x,y)=\sum_{n=0}^{\infty}\left(c_n z^n+\bar c_n\bar z^n\right)
\]
of any harmonic function, this yields the following theorem.

		\begin{theorem}
			Every graphical isotropic minimal surface can be decomposed into an infinite sum of scaled isotropic generalised Enneper surfaces and two planar minimal surfaces, one of which is parallel to the $xy$-plane.
		\end{theorem}
		
		\subsection{Decompositions of Isotropic Minimal Surfaces under the Bonnet Rotation and López–Ros Deformation}
		
		For a given Weierstrass data $(F, G)$, consider the two-parameter family of isotropic minimal surfaces $X_{\theta,\lambda}\colon \Sigma \to \mathbb{I}^{3}$ defined by applying the Bonnet rotation and the López–Ros deformation:
		\begin{equation}
			X_{\theta,\lambda}
			= \operatorname{Re} \int^{w}
			\begin{pmatrix}
				1 \\
				-i \\
				2\lambda G
			\end{pmatrix}
			\frac{e^{i\theta}}{\lambda} F \, dw.
		\end{equation}
		Here, $\theta \in (0, 2\pi]$ is the Bonnet angle, and $\lambda \in (0, +\infty)$ is the parameter for the López–Ros deformation. It is well known that $X_{\theta,\lambda}$ is a zero mean curvature (ZMC) surface in $\mathbb{I}^{3}$(see \cite{AkamineFujino2022} for more details).\\
		\noindent
		Recall that for any graphical isotropic minimal surface, there exists Weierstrass data of the form $(1, G)$ (see Lemma \ref{thm: weierstrass data of isotropic minimal surface}) whose image coincides with the surface. In this case, setting $w = u + iv$, we obtain:
		\[
		X_{\theta,\lambda}
		= \frac{1}{\lambda} \operatorname{Re} \left( e^{i\theta} \int^{w}
		\begin{pmatrix}
			1 \\
			-i \\
			2\lambda G
		\end{pmatrix}
		dw \right)
		= \frac{1}{\lambda}
		\begin{pmatrix}
			u \cos \theta - v \sin \theta \\
			v \cos \theta + u \sin \theta \\
			\operatorname{Re} \int^{w} 2\lambda e^{i\theta} G \, dw
		\end{pmatrix}.
		\]
		\noindent
		The first two components form an invertible linear map $A_{\mathbb{C}} : \mathbb{C} \to \mathbb{R}^2$, given by
		\[
		A_{\mathbb{C}}(w) = \frac{1}{\lambda} \big( u \cos \theta - v \sin \theta,\; v \cos \theta + u \sin \theta \big),
		\]
		which takes a complex number $w = u+iv$ and returns a vector in $\mathbb{R}^2$. Thus the surface is graphical, so we denote the first two components as $x$ and $y$ and the third component as $f(x,y)$. Explicitly,
		\[
		f(x,y) = \operatorname{Re} \int^{A_{\mathbb{C}}^{-1}(x,y)} 2 e^{i\theta} G \, dw.
		\]
		\noindent
		This expression can be expanded as follows:
		\[
		f(x,y) = \left( \operatorname{Re} \int^{A_{\mathbb{C}}^{-1}(x,y)} 2G \, dw \right) \cos \theta
		+ \left( \operatorname{Re} \left( i \int^{A_{\mathbb{C}}^{-1}(x,y)} 2G \, dw \right) \right) \sin \theta.
		\]
		\noindent
		Recall that $\operatorname{Re}(iz) = -\operatorname{Im}(z)$. Using this identity, we obtain
		\begin{equation}\label{eqn: decomposition theta lambda}
			f(x,y) = \left( \operatorname{Re} \int^{A_{\mathbb{C}}^{-1}(x,y)} 2G \, dw \right) \cos \theta
			- \left( \operatorname{Im} \int^{A_{\mathbb{C}}^{-1}(x,y)} 2G \, dw \right) \sin \theta.
		\end{equation}
		\noindent
		Hence, when the Weierstrass data is $(1, G)$, the Weierstrass–Enneper representation together with the Bonnet rotation and the López–Ros deformation yields a graphical isotropic minimal surface. Consequently, the third component gives the height function of such a surface.\\
		\noindent
		First, consider the assumption that the height function is of the form $\int^w 2G \, dw = \ln(f(w))$. Under this assumption, we have the following theorem.
		
		\begin{theorem}\label{thm: weierstrass prod decomposition with theta, lambda}
			Let $\{a_{n}\}_{n \geq 1}$ be a sequence of nonzero complex numbers, and let $f(z)$ be an entire function with zeros at $\{a_{n}\}$ listed according to their multiplicities. Suppose $f$ has a zero of order $k \geq 0$ at the origin. Let $X$ be a graphical isotropic minimal surface with height function $h(x,y) = \operatorname{Re}(\ln f(z))$. If we apply a Bonnet rotation with angle $\theta$ and a López–Ros deformation with parameter $\lambda$, then the resulting isotropic minimal surface is also graphical, with height function $h_{\theta,\lambda}$. Moreover, $h_{\theta,\lambda}$ admits a decomposition on a simply connected domain excluding $\{a_{n}\}_{n \geq 1}$ and $0$ into the following components:
			
			\begin{enumerate}
				\item \textbf{Helicoid contributions (from the $\sin\theta$ terms):}
				\begin{enumerate}
					\item A finite sum representing $k$ helicoids, dilated and rotated by the linear transformation $A^{-1}$, with the height function scaled by $-\sin\theta$, given by the formula  $-\sin\theta \, \left[h_{\text{Hel}} \circ A^{-1}\right] (x,y)$, where $h_{\text{Hel}}(x,y)=\tan^{-1}(y/x)$. Here $A:\mathbb{R}^2\to\mathbb{R}^2$ is the linear transformation corresponding to $A_{\mathbb{C}}:\mathbb{C}\to\mathbb{R}^2$:
					\[
					A = \frac{1}{\lambda}\begin{pmatrix} \cos \theta & -\sin \theta \\ \sin \theta & \cos \theta \end{pmatrix}.
					\]
					\item An infinite sum of dilated, rotated, and translated helicoids, with the height function scaled by $-\sin\theta$, where each term is $-\sin\theta \, \left[h_{\text{dil.Hel}} \circ A^{-1}\right] (x,y)$ with
					\[
					h_{\text{dil.Hel}}(x,y)=\tan^{-1}\!\left( \frac{a_n^2 x - a_n^1 y}{|a_n|^2 - a_n^1 x - a_n^2 y} \right),
					\]
					associated with the nonzero zeros $a_n = a_n^1 + i a_n^2$ of $f$.
					\item An infinite sum of graphical isotropic minimal surfaces given by scaled harmonic polynomials $-\sin\theta \operatorname{Im}\!\left(P_n\!\left(\tfrac{1}{a_n}Z\right)\right)$, where $Z = A_{\mathbb{C}}^{-1}(x,y)$ and $P_n$ is derived from the Weierstrass primary factors in the product formula for $f$.
					\item A graphical isotropic minimal surface given by the scaled harmonic function $-\sin\theta \operatorname{Im}(g(Z))$, where $g(z)$ is the entire function appearing in the Weierstrass product formula for $f$.
				\end{enumerate}
				
				\item \textbf{Logarithmoid of revolution contributions (from the $\cos\theta$ terms):}
				\begin{enumerate}
					\item A finite sum representing $k$ logarithmoids of revolution, dilated and rotated by the linear transformation $A^{-1}$, with the height function scaled by $\cos\theta$, given by $\cos\theta \, \left[h_{\text{LogRev}} \circ A^{-1}\right] (x,y)$, where $h_{\text{LogRev}}(x,y)=\ln\!\left(\sqrt{x^2+y^2}\right)$.
					\item An infinite sum of dilated, rotated, and translated logarithmoids of revolution, with the height function scaled by $\cos\theta$, where each term is given by $\cos\theta \, \left[h_{\text{dil.LogRev}} \circ A^{-1}\right] (x,y)$ with
					\[
					h_{\text{dil.LogRev}}(x,y)=\ln\!\left(\big|1-\tfrac{x+iy}{a_n^1 + i a_n^2}\big|\right),
					\]
					associated with the nonzero zeros $a_n = a_n^1 + i a_n^2$ of $f$.
					\item An infinite sum of graphical isotropic minimal surfaces given by scaled harmonic polynomials $\cos\theta \operatorname{Re}\!\left(P_n\!\left(\tfrac{1}{a_n}Z\right)\right)$, where $Z = A_{\mathbb{C}}^{-1}(x,y)$.
					\item A graphical isotropic minimal surface given by the scaled harmonic function $\cos\theta \operatorname{Re}(g(Z))$, where $Z = A_{\mathbb{C}}^{-1}(x,y)$.
				\end{enumerate}
			\end{enumerate}
			\noindent
			Furthermore, because the Weierstrass product representation of an entire function is not unique, the corresponding decomposition of the isotropic minimal surface is likewise not uniquely determined.
		\end{theorem}
		
	    \begin{proof}
		   The proof of this proposition follows from equation \eqref{eqn: decomposition theta lambda}. Observe that $\left(\operatorname{Re} \int^w 2G \, dw\right)$ is the height function of a graphical isotropic minimal surface and therefore admits a decomposition as given in Theorem \ref{thm: weierstrass prod decomposition 1}. Similarly, $\left(\operatorname{Im} \int^w 2G \, dw\right)$ is the height function of a graphical isotropic minimal surface and admits a decomposition as given in Theorem \ref{thm: weierstrass prod decomposition 2}. 
		
		   These decompositions are then scaled by $\cos\theta$ and $-\sin\theta$, respectively, and composed with $A^{-1}(x,y)$ which represents a scaling by $\lambda$ followed by an orthogonal transformation. Combining these ingredients yields the desired decomposition.
	   \end{proof}
		\noindent
		As an immediate corollary, consider the case where the height function is of the form $\ln(f(z))$ and the Weierstrass product formula for $f(z)$ contains no Weierstrass primary factors. We then obtain the following result.
		
		\begin{corollary}
			Let $\{a_{n}\}_{n \geq 1}$ be a sequence of nonzero complex numbers such that the series $\sum_{n=1}^{\infty} |a_n|^{-1} < \infty$. Then the function $f(z) = z^k \prod_{n=1}^{\infty} (1 - z/a_n)$ is a well-defined analytic function on the unit disc $\mathbb{D}$. Consequently, the functions $h$ given by $h(x,y) = \operatorname{Re}(\ln f(z))$ and $h(x,y) = \operatorname{Im}(\ln f(z))$ define graphical isotropic minimal surfaces on $\mathbb{D}$. \\
			\noindent
			If we apply a Bonnet rotation with angle $\theta$ and a López–Ros deformation with parameter $\lambda$, then the resulting isotropic minimal surface is also graphical, with height function $h_{\theta,\lambda}$. Moreover, $h_{\theta,\lambda}$ admits a decomposition on a simply connected domain excluding $\{a_{n}\}_{n \geq 1}$ and $0$ into the following components:
			
			\begin{enumerate}
				\item \textbf{Helicoid contributions (from the $\sin\theta$ terms):}
				\begin{enumerate}
					\item A finite sum representing $k$ helicoids, dilated and rotated by the linear transformation $A^{-1}$, with the height function scaled by $-\sin\theta$, given by the formula $-\sin\theta \, \left[h_{\text{Hel}} \circ A^{-1}\right] (x,y)$, where $h_{\text{Hel}}(x,y)=\tan^{-1}(y/x)$.
					\item An infinite sum of dilated, rotated, and translated helicoids, with the height function scaled by $-\sin\theta$, where each term is $-\sin\theta \, \left[h_{\text{dil.Hel}} \circ A^{-1}\right] (x,y)$.
				\end{enumerate}
				
				\item \textbf{Logarithmoid of revolution contributions (from the $\cos\theta$ terms):}
				\begin{enumerate}
					\item A finite sum representing $k$ logarithmoids of revolution, dilated and rotated by the linear transformation $A^{-1}$, with the height function scaled by $\cos\theta$, given by $\cos\theta \, \left[h_{\text{LogRev}} \circ A^{-1}\right] (x,y)$, where $h_{\text{LogRev}}(x,y)=\ln\!\left(\sqrt{x^2+y^2}\right)$.
					\item An infinite sum of dilated, rotated, and translated logarithmoids of revolution, with the height function scaled by $\cos\theta$, where each term is $\cos\theta \, \left[h_{\text{dil.LogRev}} \circ A^{-1}\right] (x,y)$.
				\end{enumerate}
			\end{enumerate}
			\noindent
			In both cases, dilation is contributed by both the linear transformation $A^{-1}$ (which encodes scaling by $\lambda$) and the zeros $\{a_n\}$ of $f$, while rotation is contributed solely by $A^{-1}$ (which encodes rotation by $\theta$). The translations are determined by the zeros $\{a_n\}$.
		\end{corollary}
		\noindent
		Next, one can also consider the effect of the Bonnet rotation and López–Ros transformation on the power series decomposition of isotropic minimal surfaces. Recall that in this type of decomposition, the key object is the Enneper surface. Starting from equation \eqref{eqn: decomposition theta lambda}, we can rewrite it as
		\[
		f(x,y)= \left(\operatorname{Re} \int^{A_{\mathbb{C}}^{-1}(x,y)} 2G \, dw\right) \cos \theta + \left(\operatorname{Re}\; i\int^{A_{\mathbb{C}}^{-1}(x,y)} 2G \, dw\right) \sin \theta.
		\]
		\noindent
		Consider the power series expansion of \(\int^z 2G \, dw = \sum_{n=0}^{\infty} c_n z^n\). From this we obtain
		\[
		f(x,y)= \cos \theta \cdot \sum_{n=0}^{\infty} \left(\frac{c_n z^n+\bar c_n \bar z^n}{2}\right)_{A_{\mathbb{C}}^{-1}(x,y)} + \sin \theta \cdot \sum_{n=0}^{\infty} \left(\frac{ic_n z^n-i\bar c_n \bar z^n}{2}\right)_{A_{\mathbb{C}}^{-1}(x,y)}.
		\]
		\noindent
		Using the identity \(iz^n = \left(\exp(\tfrac{i\pi}{2n})\cdot z\right)^n\) for \(n \geq 1\) (with the $n=0$ case handled separately), this can be rewritten as
		
		\begin{equation}\label{eqn: pow ser decomp theta lambda}
			f(x,y)= \cos \theta \cdot \sum_{n=0}^{\infty} \left(\frac{c_n z^n+\bar c_n \bar z^n}{2}\right)_{A_{\mathbb{C}}^{-1}(x,y)} + \sin \theta \cdot \sum_{n=0}^{\infty} \left(\frac{c_n z^n+\bar c_n \bar z^n}{2}\right)_{\exp(\frac{i\pi}{2n})\cdot A_{\mathbb{C}}^{-1}(x,y)},
		\end{equation}
		where for $n=0$ the factor $\exp(\frac{i\pi}{2n})$ is interpreted as $1$ (since $z^0$ is constant and rotation has no effect). \\
		\noindent
		Notice that \(\left(\frac{c_n z^n+\bar c_n \bar z^n}{2}\right)\) represents the height function of an isotropic generalised Enneper surface, as seen in the discussion of the previous subsection on decomposition by power series. Composing this with \(A_{\mathbb{C}}^{-1}(x,y)\) yields a rotated and dilated version of the isotropic generalised Enneper surface, which we denote by \(h_{\text{dil.Enpr}(n)}\). Similarly, when we compose with \(\exp(\frac{i\pi}{2n})\cdot A_{\mathbb{C}}^{-1}(x,y)\), we obtain a rotated and dilated version of the isotropic generalised Enneper surface with an additional rotation by \(\exp(\frac{i\pi}{2n})\), which is precisely the conjugate surface. We denote this by \(h^*_{\text{dil.Enpr}(n)}\). With this, we have the following theorem.
		
		\begin{theorem}\label{thm: power series decomposition theta lambda}
			Let $X$ be a graphical isotropic minimal surface with height function $h(x,y)$. Then $h$ must be harmonic, and suppose it admits a power series representation of the form given in \eqref{eqn: power series rep}. If we apply a Bonnet rotation with angle $\theta$ and a López–Ros deformation with parameter $\lambda$, then the resulting isotropic minimal surface is also graphical, with height function $h_{\theta,\lambda}$ (as established earlier). Moreover, $h_{\theta,\lambda}$ admits a decomposition into the following components:
			
			\begin{enumerate}
				\item A plane parallel to the $xy$-plane given by \(\operatorname{Re} \left(\frac{c_0\cos \theta + i c_0 \sin \theta}{2}\right)\).
				
				\item A plane given by \(\operatorname{Re} \left(\frac{c_1 \cdot A_{\mathbb{C}}^{-1}(x,y)(\cos \theta + i \sin \theta)}{2}\right) = \operatorname{Re} \left(\frac{c_1 \cdot A_{\mathbb{C}}^{-1}(x,y) \exp(i\theta)}{2}\right)\).
				
				\item An infinite sum of dilated and rotated isotropic generalised Enneper surfaces \(h_{\text{dil.Enpr}(n)}\) (\(n \geq 2\)), with the height function scaled by $\cos\theta$.
				
				\item An infinite sum of dilated and rotated conjugate surfaces associated to the isotropic generalised Enneper surfaces (which are themselves isotropic generalised Enneper surfaces) given by \(h^*_{\text{dil.Enpr}(n)}\) (\(n \geq 2\)), with the height function scaled by $\sin\theta$.
			\end{enumerate}
		\end{theorem}
		
		\begin{proof}
			Equation \eqref{eqn: pow ser decomp theta lambda} directly gives the decomposition. The first series consists of isotropic generalised Enneper surfaces \(\frac{c_n z^n+\bar c_n \bar z^n}{2}\) composed with \(A_{\mathbb{C}}^{-1}\), which dilates and rotates the domain, and scaled by \(\cos\theta\). The second series consists of the same expressions composed with \(\exp(\frac{i\pi}{2n})\cdot A_{\mathbb{C}}^{-1}\), where the extra factor \(\exp(\frac{i\pi}{2n})\) rotates the \(n\)-th Enneper surface into its conjugate surface \(h^*_{\text{dil.Enpr}(n)}\), scaled by \(\sin\theta\). For \(n=0\), the term \(\frac{c_0}{2}\) is constant, giving a plane parallel to the \(xy\)-plane after combining both series as \(\operatorname{Re}\bigl(\frac{c_0\cos\theta + i c_0\sin\theta}{2}\bigr)\). For \(n=1\), the two series combine to produce \(\operatorname{Re}\bigl(\frac{c_1\cdot A_{\mathbb{C}}^{-1}(x,y)\exp(i\theta)}{2}\bigr)\), a plane. For \(n \geq 2\), the terms yield the required infinite sums of dilated and rotated Enneper surfaces and their conjugates. This establishes the theorem. 
		\end{proof}
			
		\section{Decompositions in  \(\mathbb{R}^3(c)\)}
		
		In this section, we discuss the decomposition of zero mean curvature (ZMC) surfaces in a one-parameter family of pseudo-Riemannian and Riemannian spaces defined on \(\mathbb{R}^3\) with a metric that depends on a real parameter \(c\). We define this family as follows:
		\[
		\mathbb{R}^3(c) = (\mathbb{R}^3, \, dx^2 + dy^2 + c \, dz^2).
		\]
		When \(c = 1\), this space is the usual Euclidean space; when \(c = -1\), it becomes the Lorentz–Minkowski space; and when \(c = 0\), it becomes the isotropic three-dimensional space.\\
		\noindent
		The ZMC surface equation for a graphical ZMC surface in this space is given by
		\[
		(1 + c f_x^2) f_{yy} - 2c f_x f_y f_{xy} + (1 + c f_y^2) f_{xx} = 0.
		\]
		\noindent
		The Weierstrass–Enneper representation for a ZMC surface in this space is
		\[
		X_{c}
		= \operatorname{Re} \int^{w}
		\begin{pmatrix}
			1 - cG^2 \\
			-i(1 + cG^2) \\
			2G
		\end{pmatrix}
		F \, dw.
		\]
		When \(c = 1\), the above represents a minimal surface; when \(c = -1\), it represents a maximal surface; and when \(c = 0\), it represents an isotropic minimal surface. (see \cite{AkamineFujino2022} for more details)

		\subsection{Decompositions under the \(c-\)family}		
		We start with a graphical isotropic minimal surface. In Lemma \ref{thm: weierstrass data of isotropic minimal surface}, we have shown that such a surface admits a Weierstrass–Enneper representation with Weierstrass data given by \((1, G)\). We apply a \(c-\)family deformation to this Graphical Isotropic Minimal Surface, and in this case, the Weierstrass–Enneper representation becomes
		\begin{equation}\label{eqn: c family weierstrass enneper representation}			
			S_c =
			\operatorname{Re} \int^{w}
			\begin{pmatrix}
				1 - cG^2 \\
				-i(1 + cG^2) \\
				2G
			\end{pmatrix}
			\, dw
			=
			\begin{pmatrix}
				X(z,\bar z) \\
				Y(z,\bar z) \\
				Z(z,\bar z)
			\end{pmatrix}.
		\end{equation}
		\noindent
		Under this \(c\)-family deformation, the surface is again locally graphical at \((0,0)\) if the map
		\[
		\phi: (z,\bar z) \longrightarrow (X+iY,\; X-iY)(z,\bar z)
		\]
		is locally invertible at \((0,0)\). One way to ensure this is to apply the inverse function theorem and verify that the determinant of the Jacobian matrix of the above map is nonzero:
		\[
		\det
		\begin{pmatrix}
			(X+iY)_z & (X+iY)_{\bar z} \\
			(X-iY)_z & (X-iY)_{\bar z}
		\end{pmatrix}\Bigg|_{(0,0)}
		=
		\det
		\begin{pmatrix}
			1 & -c\,\overline{G^2} \\
			-cG^2 & 1
		\end{pmatrix}\Bigg|_{(0,0)}
		= 1 - c^2 |G(0,0)|^4 \neq 0.
		\]
		Hence, we have the following lemma
		
		\begin{lemma}\label{thm: graphness of c deformed family}
			Let \(S_c(z,\bar z)\) be a graphical isotropic minimal surface. As shown in Lemma \ref{thm: weierstrass data of isotropic minimal surface}, such a surface can be represented by a Weierstrass–Enneper representation with Weierstrass data \((1, G)\). If we apply a \(c\)-family deformation given by formula \eqref{eqn: c family weierstrass enneper representation}, then the resulting \(c\)-deformed ZMC surface in \(\mathbb{R}^3(c)\) with Weierstrass data \((1, G)\) is graphical near \((0,0)\) provided that \(1 - c^2 |G(0,0)|^4 \neq 0\).
		\end{lemma}
		\noindent
		We can also discuss the graphical nature of the conjugate surface to this surface. The conjugate of a \(c\)-deformed surface is given by
		\begin{equation}\label{eqn: c family conjugate weierstrass enneper representation}			
			S^*_c =
			\operatorname{Re} \int^{w}
			\begin{pmatrix}
				1 - cG^2 \\
				-i(1 + cG^2) \\
				2G
			\end{pmatrix}
			i \, dw
			=
			\begin{pmatrix}
				X^*(z,\bar z) \\
				Y^*(z,\bar z) \\
				Z^*(z,\bar z)
			\end{pmatrix}.
		\end{equation}
		We call the map which takes \((z,\bar z)\) to \((X^*+iY^*,\; X^*-iY^*)(z,\bar z)\) as 
		\[
		\phi^*: (z,\bar z) \longrightarrow (X^*+iY^*,\; X^*-iY^*)(z,\bar z).
		\]
		With this we have the following lemma.
		
		\begin{lemma}\label{thm: graphness of c deformed conjugate family}
			Let \(S^*_c(z,\bar z)\) be the conjugate isotropic minimal surface to the graphical isotropic minimal surface \(S_c(z,\bar z)\). Since the Weierstrass data of \(S_c(z,\bar z)\) is \((1, G)\), this conjugate isotropic minimal surface can be represented by a Weierstrass–Enneper representation with Weierstrass data \((i, G)\). If we apply a \(c\)-family deformation to this surface given by formula \eqref{eqn: c family conjugate weierstrass enneper representation}, then the resulting \(c\)-deformed ZMC surface in \(\mathbb{R}^3(c)\) with Weierstrass data \((i, G)\) is graphical near \((0,0)\) provided that \(1 - c^2 |G(0,0)|^4 \neq 0\). In other words, if the minimal surface \(S_c(z,\bar z)\) is graphical near \((0,0)\), then so is its conjugate minimal surface \(S^*_c(z,\bar z)\).
		\end{lemma}
		\begin{proof}
	    The proof is straightforward and follows the exact same steps as the previous lemma.
	    \end{proof}
        \noindent
	    With these lemmas, one can determine whether the \(c\)-deformed surface is graphical. Next, we assume that the map \(\phi\) is locally invertible around \((0,0)\). Now consider the \(z\)-component of this ZMC surface; it is either of the form \(\operatorname{Re}\int^w 2G \, dw\) or \(\operatorname{Im}\int^w 2G \, dw\). We suppose that \(\int^w 2G \, dw\) is of the form \(\ln(f(w))\). Then we can apply the Weierstrass product theorem to this as we did in the previous section.\\
	    \noindent
	    Thus, assuming \(\{a_{n}\}_{n \geq 1}\) is a sequence of nonzero zeros of \(f(z)\) listed according to their multiplicities, and suppose \(f\) has a zero of order \(k \geq 0\) at the origin. The Weierstrass Factorization Theorem guarantees that \(f\) can be expressed as
	    \[
	    f(z) = z^k e^{g(z)} \prod_{n=1}^{\infty} E_n\!\left( \frac{z}{a_n} \right).
	    \]
	    Taking the logarithm on both sides and extracting the real part, we obtain the following
	    \[
	    \operatorname{Re}[\ln f(z)] = k \ln|z| + \operatorname{Re}[g(z)] + \sum_{n=1}^{\infty} \ln\left| 1 - \frac{z}{a_n} \right| + \sum_{n=1}^{\infty} \operatorname{Re}\!\left[ P_n\!\left( \frac{z}{a_n} \right) \right].
	    \]
	    Similarly, taking the logarithm of both sides and extracting the imaginary part, we obtain
	    \[
	    \operatorname{Im}[\ln f(z)] = k \tan^{-1}(y/x) + \operatorname{Im}[g(z)] + \sum_{n=1}^{\infty} \tan^{-1}\!\left( \frac{a_n^2 x - a_n^1 y}{|a_n|^2 - a_n^1 x - a_n^2 y} \right) + \sum_{n=1}^{\infty} \operatorname{Im}\!\left[ P_n\!\left( \frac{z}{a_n} \right) \right].
	    \]
	  \noindent  
	    We recall the following naming conventions:
	    \[
	    h_{\text{Hel}}(x,y) = \tan^{-1}(y/x), \qquad
	    h_{\text{dil.Hel}}(x,y) = \tan^{-1}\!\left( \frac{a_n^2 x - a_n^1 y}{|a_n|^2 - a_n^1 x - a_n^2 y} \right),
	    \]
	    \[
	    h_{\text{LogRev}}(x,y) = \ln\!\left(\sqrt{x^2+y^2}\right), \qquad
	    h_{\text{dil.LogRev}}(x,y) = \ln\!\left(\big|1-\tfrac{x+iy}{a_n^1 + i a_n^2}\big|\right).
	    \] 
	    Next we note that as a graph around \((0,0)\) the height functions of \(S_c\) and \(S_c^*\) can be given as \(\operatorname{Re}[\ln f(\phi^{-1}(x,y))]\) and \(\operatorname{Im}[\ln f([\phi^*]^{-1}(x,y))]\) respectively. Now here \((x,y)\) is changed to \(\operatorname{Re}[\ln f(\phi^{-1}(x,y))]\) and \(\operatorname{Im}[\ln f([\phi^*]^{-1}(x,y))]\) respectively. With this, we have the following two theorems:		
		
		\begin{theorem}\label{thm: weierstrass prod decomposition 1 c family}
			Let $\{a_{n}\}_{n \geq 1}$ be a sequence of nonzero complex numbers, and let $f(z)$ be an entire function with zeros at $\{a_{n}\}$ listed according to their multiplicities. Suppose $f$ has a zero of order $k \geq 0$ at the origin. Let $X$ be a graphical isotropic minimal surface with height function $h(x,y) = \operatorname{Re}(\ln f(z))$. Apply a $c$-family deformation to this isotropic surface and assume that the resulting $c$-deformed ZMC surface in $\mathbb{R}^3(c)$ is again graphical, with height function denoted by $h_c(x,y)$. Then $h_c$ admits a decomposition into the following components:
			
			\begin{enumerate}
				\item A finite sum representing $k$ logarithmoids of revolution composed with $\phi^{-1}(z,\bar z)$, given by $h_{\text{LogRev}}(\phi^{-1}(z,\bar z))$;
				
				\item An infinite sum of dilated and translated logarithmoids of revolution composed with $\phi^{-1}(z,\bar z)$, given by $h_{\text{dil.LogRev}}(\phi^{-1}(z,\bar z))$
				
				\item An infinite sum of graphical isotropic minimal surfaces given by the harmonic polynomials $\operatorname{Re}(P_n(z/a_n))$ composed with $\phi^{-1}(z,\bar z)$, where $P_n(z/a_n)$ is the polynomial derived from the Weierstrass primary factors in the product formula for $f$;
				
				\item A graphical isotropic minimal surface given by the harmonic function $\operatorname{Re}(g(z))$ composed with $\phi^{-1}(z,\bar z)$, where $g(z)$ is the entire function appearing in the Weierstrass product formula for $f$.
			\end{enumerate}
			
			\noindent
			The function \(\phi: (z,\bar z) \mapsto (X+iY,\; X-iY)\) is given by
			\[
			\phi(z,\bar z) = \left( z - \overline{\int^z c G^2}, \; \bar z - \int^z c G^2 \right),
			\]
			where \(G = \dfrac{f'(z)}{2f(z)}\).
		\noindent	
			Moreover, since the Weierstrass product representation of an entire function is not unique, the corresponding decomposition of the isotropic minimal surface is likewise not uniquely determined. This decomposition is valid on a simply connected domain whose image under \(\phi\) excludes $\{a_{n}\}_{n \geq 1}$ and $0$

		\end{theorem}
		
		\begin{proof}
			The proof follows exactly the same steps as the proof of Theorem \ref{thm: weierstrass prod decomposition 1}, with the only difference being that every term is composed with $\phi^{-1}$. This composition arises because the height function of the surface is given by $\operatorname{Re}[\ln f(\phi^{-1}(z,\bar z))]$.
			
			The formula for $\phi$ given above can be easily computed from equation \eqref{eqn: c family weierstrass enneper representation}. Here $G$ is the Gauss map of the isotropic graphical minimal surface $(x, y, h(x,y))$ with $h(x,y) = \operatorname{Re} \ln(f(x+iy))$. In this case, we can apply Lemma \ref{thm: weierstrass data of isotropic minimal surface} to obtain $G = \partial_z h$, and $\partial_z h$ is equal $\dfrac{f'(z)}{2f(z)}$.
		\end{proof}

        \begin{remark}
            The components of the decomposition above need not be ZMC surface in $\mathbb{R}^3$ as $\phi^{-1}$ is an arbitrary function with no geometric interpretation in general. 
        \end{remark}
        \noindent
		Similarly, for the case where $h$ is defined by the imaginary part of $\ln f(z)$, we obtain the following result regarding its infinite sum decomposition:
		
		\begin{theorem}\label{thm: weierstrass prod decomposition 2 c family}
			Let $\{a_{n}\}_{n \geq 1}$ be a sequence of nonzero complex numbers, and let $f(z)$ be an entire function with zeros at $\{a_{n}\}$ listed according to their multiplicities. Suppose $f$ has a zero of order $k \geq 0$ at the origin. Let $X$ be a graphical isotropic minimal surface with height function $h(x,y) = \operatorname{Im}(\ln f(z))$. Apply a $c$-family deformation to this isotropic surface and assume that the resulting $c$-deformed ZMC surface in $\mathbb{R}^3(c)$ is again graphical, with height function denoted by $h_c(x,y)$. Then $h_c$ admits a decomposition into the following components:
			
			\begin{enumerate}
				\item A finite sum representing $k$ helicoids composed with $(\phi^*)^{-1}(z,\bar z)$, given by the formula $h_{\text{Hel}}((\phi^*)^{-1}(z,\bar z))$;
				
				\item An infinite sum of dilated and translated helicoids composed with $(\phi^*)^{-1}(z,\bar z)$, given by $h_{\text{dil.Hel}}((\phi^*)^{-1}(z,\bar z))$
				
				\item An infinite sum of graphical isotropic minimal surfaces given by the harmonic polynomials $\operatorname{Im}(P_n(z/a_n))$ composed with $(\phi^*)^{-1}(z,\bar z)$, where $P_n(z/a_n)$ is the polynomial derived from the Weierstrass primary factors in the product formula for $f$;
				
				\item A graphical isotropic minimal surface given by the harmonic function $\operatorname{Im}(g(z))$ composed with $(\phi^*)^{-1}(z,\bar z)$, where $g(z)$ is the entire function appearing in the Weierstrass product formula for $f$.
			\end{enumerate}
			
			\noindent
			The function \(\phi^*: (z,\bar z) \mapsto (X^*+iY^*,\; X^*-iY^*)\) is given by
			\[
			(\phi^*)(z,\bar z) = \left( iz + i\overline{\int^z c G^2}, \; -i\bar z - i \int^z c G^2 \right),
			\]
			where \(G = \dfrac{f'(z)}{2f(z)}\).
			\noindent
			Moreover, since the Weierstrass product representation of an entire function is not unique, the corresponding decomposition of the isotropic minimal surface is likewise not uniquely determined. This decomposition is valid on a simply connected domain whose image under \(\phi\) excludes $\{a_{n}\}_{n \geq 1}$ and $0$.
			\end{theorem}
		
		\begin{proof}
			The proof follows immediately by applying the argument of the previous proposition to the imaginary part of the logarithm of the Weierstrass factorization of $f(z)$.
		\end{proof}
		\noindent

In analogy with the preceding theorems (see Theorem \ref{thm: weierstrass prod decomposition 1 c family} and Theorem \ref{thm: power series decomposition theta lambda}), one can also obtain a power series decomposition. We state the theorem below; its proof follows immediately.

		\begin{theorem}\label{thm: power series decomposition c family}
			Let $X$ be a graphical isotropic minimal surface with height function $h(x,y)$. Then $h$ must be harmonic, and suppose it admits a power series representation of the form given in \eqref{eqn: power series rep}. If we apply a $c$-family deformation to this isotropic surface and assume that the resulting $c$-deformed ZMC surface in $\mathbb{R}^3(c)$ is again graphical, with height function denoted by $h_c(x,y)$, then $h_c$ admits a decomposition into the following components:
			
			\begin{enumerate}
				\item A plane parallel to the $xy$-plane given by \(\operatorname{Re} \left(c_0\right)\).
				
				\item A surface given by \(\operatorname{Re} \left(c_1 \cdot \pi_1 \circ \phi^{-1}(z,\bar z)\right)\), where \(\pi_1\) is the projection to the first component.
				
				\item An infinite sum of isotropic generalised Enneper surfaces composed with $\phi^{-1}(x,y)$.
			\end{enumerate}
		\end{theorem}
		
		\subsection{Some Observations}
		
		\begin{enumerate}
			\item From Theorems \ref{thm: weierstrass prod decomposition 1 c family}, \ref{thm: weierstrass prod decomposition 2 c family}, and \ref{thm: power series decomposition c family}, we see that the building blocks of ZMC surfaces across the $c$-family are essentially the helicoid, the logarithmoid of revolution, and the generalised isotropic Enneper surface.
			
			\item It should be noted that the helicoid given by $h(x,y) = \tan^{-1}(y/x)$ is a ZMC surface in all spaces $\mathbb{R}^3(c)$. However, the situation is different for the logarithmoid of revolution and the generalised isotropic Enneper surface. Before explaining this, we introduce the following definitions.
			
			\begin{definition}
				We define a $c$-catenoid to be a ZMC surface in $\mathbb{R}^3(c)$ with Weierstrass data $(1, \frac{1}{2z})$. Similarly, we define a generalised $c$-Enneper surface to be a ZMC surface in $\mathbb{R}^3(c)$ with Weierstrass data $(1, z^n)$.
			\end{definition}
			
			We now examine the relationship between the $c$-catenoid and the logarithmoid of revolution, as well as the relationship between the generalised $c$-Enneper surface and the generalised isotropic Enneper surface. First, consider the $c$-catenoid given by
			\[
			\operatorname{Re} \int^{w}
			\begin{pmatrix}
				1 - \frac{c}{4z^2} \\
				-i(1 + \frac{c}{4z^2}) \\
				\frac{1}{z}
			\end{pmatrix}
			dw
			=
			\begin{pmatrix}
				\frac{z}{2} + \frac{\bar z}{2} + \frac{c}{8z} + \frac{c}{8\bar z} \\[6pt]
				\frac{-iz}{2} + \frac{i\bar z}{2} + \frac{ic}{8z} - \frac{ic}{8\bar z} \\[6pt]
				\ln|z|
			\end{pmatrix}.
			\]
			
			Define $\phi_{c.\text{Cat}}(z,\bar z) = \left( \frac{z}{2} + \frac{\bar z}{2} + \frac{c}{8z} + \frac{c}{8\bar z},\; \frac{-iz}{2} + \frac{i\bar z}{2} + \frac{ic}{8z} - \frac{ic}{8\bar z} \right)$ as the first two components of the Weierstrass–Enneper representation of the $c$-catenoid. Note that the third component is the logarithmoid of revolution. The relationship between the height function $h_{c.\text{Cat}}$ of the graph of the $c$-catenoid and the height function $h_{\text{LogRev}}$ of the graph of the logarithmoid of revolution can then be expressed as
			\begin{equation}\label{eqn: rel logrev c-cat}
				h_{c.\text{Cat}} = h_{\text{LogRev}} \circ \phi_{c.\text{Cat}}^{-1}.
			\end{equation}
			
			Through a similar calculation, one can find the relationship between the height function $h_{c.\text{Enpr}(n)}$ of the graph of the generalised $c$-Enneper surface and the height function $h_{\text{Enpr}(n)}$ of the graph of the generalised isotropic Enneper surface. Consider the generalised $c$-Enneper surface given by
			\[
			\operatorname{Re} \int^{w}
			\begin{pmatrix}
				1 - cz^n \\
				-i(1 + cz^n) \\
				2z^n
			\end{pmatrix}
			dw
			=
			\begin{pmatrix}
				\frac{z}{2} + \frac{\bar z}{2} - \frac{cz^{n+1}}{2(n+1)} - \frac{c\bar z^{n+1}}{2(n+1)} \\[6pt]
				\frac{-iz}{2} + \frac{i\bar z}{2} - \frac{icz^{n+1}}{2(n+1)} + \frac{ic\bar z^{n+1}}{2(n+1)} \\[6pt]
				\frac{2z^{n+1}}{n+1}
			\end{pmatrix}.
			\]
			
			Define $\phi_{c.\text{Enpr}}(z,\bar z) = \left( \frac{z}{2} + \frac{\bar z}{2} - \frac{cz^{n+1}}{2(n+1)} - \frac{c\bar z^{n+1}}{2(n+1)},\; \frac{-iz}{2} + \frac{i\bar z}{2} - \frac{icz^{n+1}}{2(n+1)} + \frac{ic\bar z^{n+1}}{2(n+1)} \right)$. With this, we obtain the relation
			\begin{equation}\label{eqn: rel iso-enpr c-enpr}
				h_{c.\text{Enpr}(n)} = h_{\text{Enpr}(n)} \circ \phi_{c.\text{Enpr}}^{-1},
			\end{equation}
			where $n$ is the degree of the Gauss map of the generalised $c$-Enneper surface.
			
			Using equations \eqref{eqn: rel logrev c-cat} and \eqref{eqn: rel iso-enpr c-enpr}, together with the fact that helicoids are ZMC surfaces in $\mathbb{R}^3(c)$ for every value of $c$, we conclude that the the helicoid, the $c$-catenoids, and the generalised $c$-Enneper surfaces can also be considered as the building blocks in the decomposition theory of ZMC surfaces of $\mathbb{R}^3(c)$.
			
			\item If, in Theorems \ref{thm: weierstrass prod decomposition 1 c family} and \ref{thm: weierstrass prod decomposition 2 c family}, we assume that the function $f$ admits a Weierstrass factorisation without any nonlinear polynomial in its primary factors, then it decomposes either into a purely infinite sum of logarithmoids of revolution composed with some function, or into a purely infinite sum of helicoids again composed with some function. The function with which the composition is performed is the same as that appearing in Theorems \ref{thm: weierstrass prod decomposition 1 c family} and \ref{thm: weierstrass prod decomposition 2 c family}.
			
			\item One can also apply the Bonnet rotation and López–Ros transformation to this $c$-family of ZMC surfaces using the formula
			\begin{equation}\label{eqn: weierstrass enneper c theta lambda}
				S_{\theta,\lambda,c}
				= \operatorname{Re} \int^{w}
				\begin{pmatrix}
					1 - c\lambda^{2}G^{2} \\
					-i(1 + c\lambda^{2}G^{2}) \\
					2\lambda G
				\end{pmatrix}
				\frac{e^{i\theta}}{\lambda} F \, dw.
			\end{equation}
			One can take \(F = 1\) if we assume that when \(c = 0\) (the isotropic case) the surface is graphical. In that case, one can easily see that the ZMC surface in \(\mathbb{R}^3(c)\) given by formula \eqref{eqn: weierstrass enneper c theta lambda} forms a graph near \((0,0)\) provided that \(1 - c^2 \lambda^4 |G(0,0)|^4 \neq 0\). The proof follows exactly the same steps as Lemma \ref{thm: graphness of c deformed family}.
			
			One can also obtain decomposition results similar to Theorem \ref{thm: weierstrass prod decomposition with theta, lambda}, where some terms are helicoids composed with some function or other helicoid-related terms, and other terms are logarithmoids of revolution composed with some function or other logarithmoid-of-revolution-related terms. A decomposition based on the power series representation, analogous to Theorem \ref{thm: power series decomposition theta lambda}, can also be given, where the terms are essentially generalised isotropic Enneper surfaces and planes. We omit the details here to avoid repetition and complexity in the presentation.
		\end{enumerate}

\section{Finite Decompositions of Zero Mean Curvature Surfaces}
	
	In the last four sections, we have given a theory of infinite decompositions of ZMC surfaces in various spaces. In this section we will give a finite decomposition version of this theory. A related work has been done in \cite{DeyGhoshSoundararajan2023}. We will make a case for this finite decomposition by showing the finite decomposition version of the example of catenoid and Scherk's tower in various spaces. \\
	\noindent
	First let us start with the conjugate Euler--Ramanujan identity \ref{eqn: ER type ScC identity}, which is the basis of all the identities we obtained in those concrete examples:
	\[
	(x^2+y^2)\cdot\prod_{k\neq 0}\left[\frac{(x+k\pi)^2+y^2}{\pi^2 k^2}\right]=\sin^2(x)+\sinh^2(y).
	\]
	Let $m$ be a fixed finite positive integer. Then any integer $n$ can be written as $n = m k_n + j$, $j=0,1,\dots,m-1$, with $k_n \in \mathbb{Z}$. At this point we state a useful fact which will be useful in the rearrangement of the terms in the infinite product.\\
	\noindent
	Suppose that $A=\prod_{n=-\infty}^{\infty}A_n$ is a convergent infinite product, and also suppose that $A^j=\prod_{k_n=-\infty}^{\infty}A_{mk_n+j}$ is also a convergent infinite subproduct of $A$ for $j=0,1,\dots,m-1$. Then we have that $A=\prod_{j=0}^{m-1}A^j$. With this we can decompose the conjugate Euler--Ramanujan identity \ref{eqn: ER type ScC identity} as follows.
	\[
	(x^2+y^2)\cdot\prod_{k\neq 0}\left[\tfrac{(x+k\pi)^2+y^2}{\pi^2 k^2}\right]=(x^2+y^2)\cdot\prod_{k_n\neq 0}\left[\tfrac{(x+mk_n\pi)^2+y^2}{\pi^2 m^2k_n^2}\right]\cdot\prod_{j=1}^{m-1}\left[\prod_{k_n \in \mathbb{Z}} \left[\tfrac{(x+ (mk_n+ j)\pi)^2+y^2}{\pi^2 (mk_n + j)^2}\right]\right].
	\]
	The only thing to check for the above formula to work is whether each subproduct is convergent or not.\\
	\noindent
	Let us start with the first subproduct:
	\[
	(x^2+y^2)\cdot\prod_{k_n\neq 0}\left[\frac{(x+mk_n\pi)^2+y^2}{\pi^2 m^2k_n^2}\right]= m^2 \left[(\tfrac{x}{m})^2 + (\tfrac{y}{m})^2\right] \prod_{k_n \neq 0} \left[ \frac{(\tfrac{x}{m} + k_n \pi)^2 + (\tfrac{y}{m})^2}{\pi^2 k_n^2 } \right].
	\]
	We know from the conjugate Euler--Ramanujan identity that the RHS of the above equation is convergent and is equal to \(m^2 \left[\sin^2(\tfrac{x}{m})+\sinh^2(\tfrac{y}{m})\right]\).\\
	\noindent
	Next consider the \(j^{th}\) subproduct.
	\small
	\begin{align*}
		\prod_{k_n \in \mathbb{Z}} \left[\frac{(x+ (mk_n+ j)\pi)^2+y^2}{\pi^2 (mk_n + j)^2}\right]&=\frac{(x+ j\pi)^2+y^2}{\pi^2 j^2}\prod_{k_n \neq 0} \left[\frac{(x+ (mk_n+ j)\pi)^2+y^2}{\pi^2 m^2 k_n^2}\right] \prod_{k_n \neq 0} \left[\frac{\pi^2 m^2 k_n^2}{\pi^2 (mk_n + j)^2}\right]\\
		&=\frac{\big(\tfrac{x+j\pi}{m}\big)^2+\big(\tfrac{y}{m}\big)^2}{\pi^2 (\tfrac{j}{m})^2} \prod_{k_n \neq 0} \left[\frac{\big(\tfrac{x+j\pi}{m}+k_n\pi\big)^2+\big(\tfrac{y}{m}\big)^2}{\pi^2 k_n^2} \right] \prod_{k_n \neq 0} \left[\frac{m^2 k_n^2}{(mk_n + j)^2}\right].
	\end{align*}
	\normalsize
	In the first step we applied a rearrangement to write the RHS of the first step into two infinite products. We use the fact that if \(A=\prod_{n=-\infty}^{\infty}A_n\) and \(B=\prod_{n=-\infty}^{\infty}B_n\) are convergent infinite products, then \(C=\prod_{n=-\infty}^{\infty}A_n B_n\) is a convergent infinite product and \(C=AB\). We know from the conjugate Euler--Ramanujan identity that the first subproduct in the RHS of the above equation is convergent and is equal to \(\frac{m^2}{\pi^2 j^2} \left[\sin^2(\tfrac{x+j\pi}{m})+\sinh^2(\tfrac{y}{m})\right]\). The second term is the well known \(j^{th}\) \textbf{\textit{Wallis Product}}; we will discuss this more and its derivation in the Appendix \ref{sec: wallis product}, Theorem \ref{thm: Wallis product}. But for now we will just state the equation: \(\prod_{n \neq 0} \left[ \frac{ n^2 }{ (n + j/m)^2}\right] = \left[\frac{\pi j }{m \sin(\pi j/m )} \right]^2\).\\
	\noindent
	Combining all the things, we find that the \(i^{th}\) product is convergent and its value is equal to \(\frac{\sin^2(\tfrac{x+j\pi}{m})+\sin^2(\tfrac{y}{m})}{\sin^2(\pi j/m )}\). \\
	\noindent
	Hence, from all the discussion above, we get:
	\begin{equation}\label{eqn: finite Conj ER}
		\sin^2(x)+\sinh^2(y)=m^2 \cdot \prod_{j=1}^{m-1}\operatorname{cosec}^2(\pi j/m) \cdot \prod_{j=0}^{m-1}\left[\sin^2(\tfrac{x+j\pi}{m})+\sinh^2(\tfrac{y}{m})\right].
	\end{equation}
	This identity is the \textbf{\textit{finite product version of the conjugate Euler--Ramanujan identity}}. Next, we replace $x$ by $x/2$ and $y$ by $y/\sqrt{2}$, to obtain:
	\[
	\sin^2(x/2)+\sinh^2(y/\sqrt{2})=m^2 \cdot \prod_{i=1}^{m-1}\operatorname{cosec}^2(\pi j/m) \cdot \prod_{j=0}^{m-1}\left[\sin^2(\tfrac{x+2j\pi}{2m})+\sinh^2(\tfrac{y}{m\sqrt{2}})\right].
	\]
	We recall that \(\cosh^2\left(\frac{f_{\text{ScherkTower}}}{\sqrt{2}}\right)=\sinh^2\left(\frac{y}{\sqrt{2}}\right)+2\sin^2\left(\frac{x}{2}\right).\) Using this we can see that the LHS of the above identity is \(\frac{1}{2}\left[\cosh^2\left(\frac{f_{\text{ScherkTower}}(x,y)}{\sqrt{2}}\right)+\cosh^2\left(\frac{f_{\text{ScherkTower}}(0,y)}{\sqrt{2}}\right)\right]\). Now we apply a translation and dilation on the Scherk's tower given by \((x,y)\) going to \((\tfrac{x+2j\pi}{m},\tfrac{y}{m})\) and denote the resulting dilated and translated surface as \(f^j_{\text{Dil.ScherkTower}}(x,y)\). Then the product on the RHS is going to be: \(\frac{1}{2}\left[\cosh^2\left(\frac{f^j_{\text{Dil.ScherkTower}}(x,y)}{\sqrt{2}}\right)+\cosh^2\left(\frac{f^j_{\text{Dil.ScherkTower}}(0,y)}{\sqrt{2}}\right)\right]\). When we apply this to equation \eqref{eqn: finite Conj ER} we obtain the following theorem:
	
	\begin{theorem}{\textbf{\textit{(Finite decompositions of Scherk's tower)}}}\label{thm: Finite decomposition of scherk's tower}
    The identity
		\begin{align}
			&\tfrac{1}{2}\left[\cosh^2\left(\tfrac{f_{\text{ScherkTower}}(x,y)}{\sqrt{2}}\right)+\cosh^2\left(\tfrac{f_{\text{ScherkTower}}(0,y)}{\sqrt{2}}\right)\right] \nonumber \\
			&\quad = m^2 \cdot \prod_{j=1}^{m-1}\operatorname{cosec}^2(\pi j/m) \cdot \prod_{j=0}^{m-1}\left[\tfrac{1}{2}\left(\cosh^2\left(\tfrac{f^j_{\text{Dil.ScherkTower}}(x,y)}{\sqrt{2}}\right)+\cosh^2\left(\tfrac{f^j_{\text{Dil.ScherkTower}}(0,y)}{\sqrt{2}}\right)\right)\right]
		\end{align}
        is valid on the domain \(D=\{(x,y)|\sinh^2{\tfrac{y}{m\sqrt2}}>1\}\) (see Figure 3).
	\end{theorem}
    \begin{proof}
	 The domain of the \(f_{\text{ScherkTower}}\) is clearly lying inside the domain \(D\), because the domain of \(f_{\text{ScherkTower}}\) is \(\sin^2\left(\frac{x}{2}\right) + \sinh^2\left(\frac{y}{\sqrt{2}}\right) > 1\), which happens also if \(\sinh^2\left(\frac{y}{\sqrt{2}}\right) > 1\). Secondly the domain of \(f_{\text{Dil.ScherkTower}}\) is \(\sin^2\left(\frac{x+2j\pi}{2m}\right) + \sinh^2\left(\frac{y}{m\sqrt{2}}\right) > 1\), which happens also if \(\sinh^2\left(\frac{y}{m\sqrt{2}}\right) > 1\). So this condition \(\sinh^2\left(\frac{y}{m\sqrt{2}}\right) > 1\) is a sufficient domain for well-definedness of this equation. In fact, notice that for all pairs \((x,y)\) for which \(\sinh^2\left(\frac{y}{\sqrt{2}}\right) < 1\), we have that \(f_{\text{ScherkTower}}(0,y)\) and \(f^0_{\text{Dil.ScherkTower}}(x,y)\) are not well defined, as the domain of \(\cosh^{-1}\) does not contain a value less than \(1\). Hence, this condition \(\sinh^2\left(\frac{y}{m\sqrt{2}}\right) > 1\) is a necessary domain for well-definedness of the finite decompositions of Scherk's tower.
     \end{proof}

    \begin{figure}[htbp]
	\centering
	\begin{subfigure}{0.45\textwidth}
		\centering
		\includegraphics[width=\linewidth]{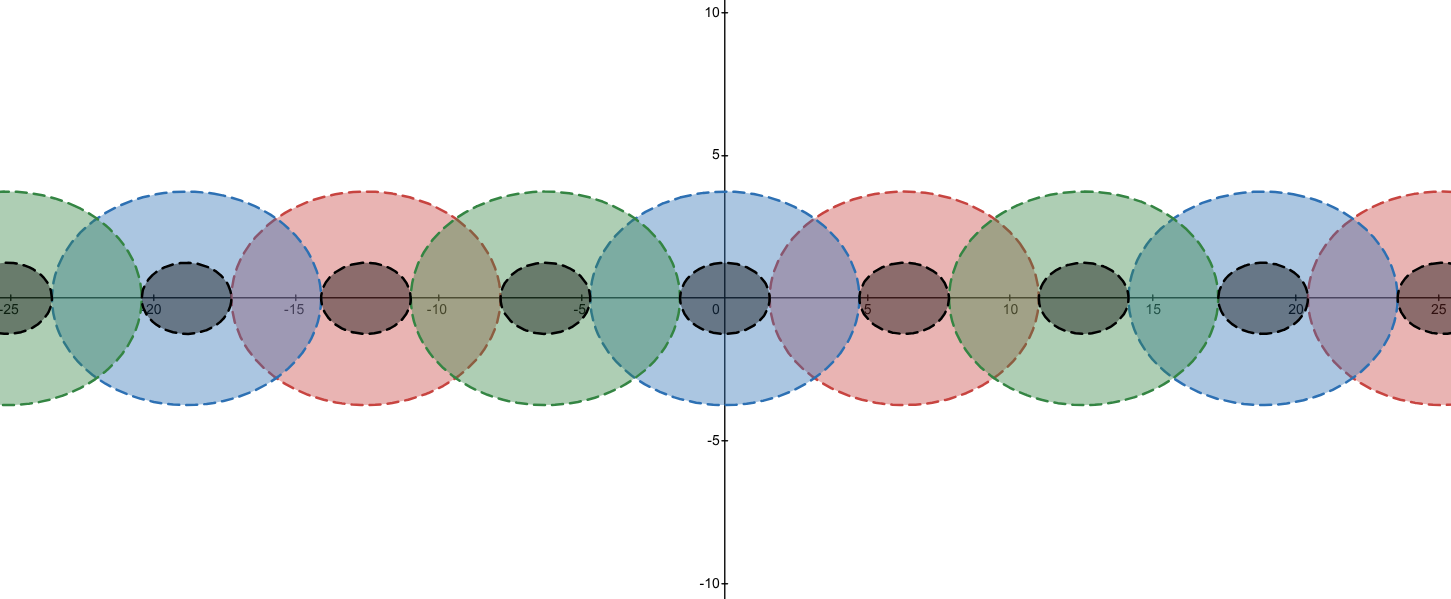}
		\caption{\footnotesize Regions where Dilated and undilated Scherk's Tower are ill defined\normalsize}
	\end{subfigure}
	\hfill
	\begin{subfigure}{0.45\textwidth}
		\centering
		\includegraphics[width=\linewidth]{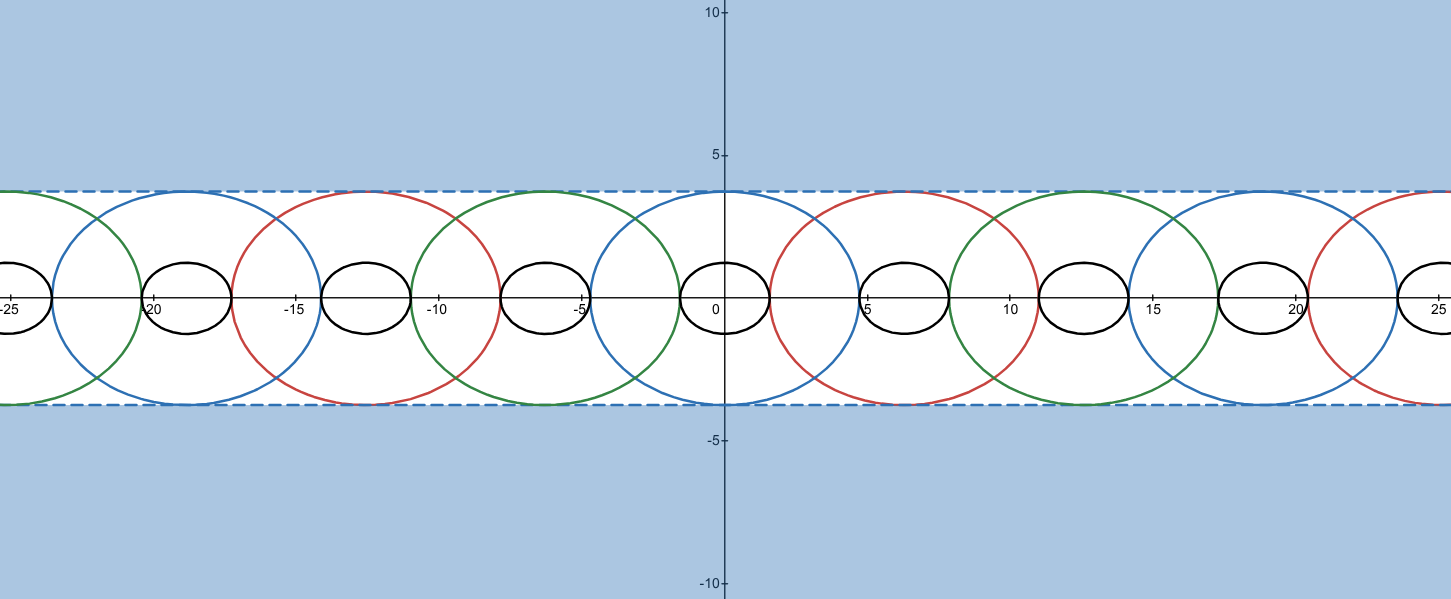}
		\caption{\footnotesize Effective Domain of the finite decomposition of Scherk's tower\normalsize}
	\end{subfigure}
	\caption{\footnotesize The above two pictures depict various aspects of the domain of finite decompositions of Scherk's tower. Here we fixed \(m=3\). The black shaded region in subfigure (a) gives the region where \(f_{\text{ScherkTower}}\) is ill defined, i.e., where it shows hole behaviour. The blue shaded elliptical shaped region in the same subfigure indicates the region where \(f^0_{\text{Dil.ScherkTower}}\) is ill defined; similarly, the red shaded region shows the ill defined region for \(f^1_{\text{Dil.ScherkTower}}\) and green for \(f^2_{\text{Dil.ScherkTower}}\). Subfigure (b) shows the blue shaded region above a line passing just above and below these elliptical shaped regions of ill defined domains and gives us the effective domain of the finite decomposition equation of Scherk's Tower given in Theorem \ref{thm: Finite decomposition of scherk's tower}.\normalsize}
	\label{pic: Domain Analysis Catenoid Scherk's Tower Identity}
\end{figure}
	\noindent
 Similarly one can also think about the finite decompositions of Scherk's surface and Scherk's tower in isotropic 3-space. We recall that the isotropic Scherk's tower is given by \(\operatorname{Re}(\ln(\sin(x+iy)))\) and the isotropic Scherk's surface is given by \(\operatorname{Im}(\ln(\sin(x+iy)))\). We start with the arguments of Theorem \ref{isotropic Scherk's id}.
\noindent	
	Recall that \(f^{\mathbb{I}^3}_{\mathcal{S}_{\text{tower}}}(x,y)=\operatorname{Re}(\ln(\sin(x+iy)))=\sum_{k \neq 0} \ln \sqrt{\tfrac{(x+k\pi)^2 + y^2}{k^2\pi^2}} + \ln \sqrt{x^2 + y^2}\), which can also be given by \(\ln\left[\sqrt{\sin^2(x)+\sinh^2(y)}\right]\). But from the finite product version of the conjugate Euler--Ramanujan identity \ref{eqn: finite Conj ER} we have that 
	\[
	\ln\left[\sqrt{\sin^2(x)+\sinh^2(y)}\right]=c+\sum_{j=0}^{m-1}\ln\left[\sqrt{\sin^2(\tfrac{x+j\pi}{m})+\sinh^2(\tfrac{y}{m})}\right],
	\]
	where \(c=\ln\left[\sqrt{m^2 \cdot \prod_{j=1}^{m-1}\operatorname{cosec}^2(\pi j/m)}\right]\). Now from the definition of \(f^{\mathbb{I}^3}_{\mathcal{S}_{\text{tower}}}(x,y)\) we obtain:
	\[
	f^{\mathbb{I}^3}_{\mathcal{S}_{\text{tower}}}(x,y) = c + \sum_{j=0}^{m-1} f^{\mathbb{I}^3}_{\mathcal{S}_{\text{tower}}} ((x+ j \pi)/m, y/m).
	\]
    One thing we need to take care of is the domain of this identity. We know that the logarithm \(\ln\) is ill-defined at \(0\), and hence the points \((x,y) = (n\pi, 0)\) for any \(n \in \mathbb{Z}\) are precisely the points where the left-hand side of this identity is not defined. Similarly, for the right-hand side, in the \(j\)-th term, the logarithm is ill-defined for the points \((x,y) = ((nm - j)\pi, 0)\) for \(j = 0, 1, 2, \dots, m\) and any \(n \in \mathbb{Z}\). But these points are precisely of the form \((x,y) = (n\pi, 0)\). Hence, for all points of the form \((x,y) = (n\pi, 0)\), the identity is ill-defined.
    
	In an exactly similar way one can prove an identity for the isotropic Scherk's surface. We give both of these results as the following theorem:
	
	\begin{theorem}\label{isotropic Scherk's id --FD}
	\vspace{1em}
	(a) \noindent \textbf{Finite decomposition of Isotropic Scherk Tower:}
	\[
	f^{\mathbb{I}^3}_{\mathcal{S}_{\text{tower}}}(x,y) = c + \sum_{j=0}^{m-1} f^{\mathbb{I}^3}_{\mathcal{S}_{\text{tower}}} \left(\frac{x + j\pi}{m}, \frac{y}{m}\right),
	\]
	where the domain is any simply connected domain that does not contain points of the form \((n\pi, 0)\) for \(n \in \mathbb{Z}\).

	(b) \noindent \textbf{Finite decomposition of Isotropic Scherk Surface:}
	\[
	f^{\mathbb{I}^3}_{\mathcal{S}_{\text{surf}}}(x,y) = \sum_{j=0}^{m-1} f^{\mathbb{I}^3}_{\mathcal{S}_{\text{surf}}}\left(\frac{x + j\pi}{m}, \frac{y}{m}\right),
	\]
	defined on an appropriate simply connected domain where the Scherk surface and the dilated Scherk surfaces are well defined, where 
	\[
	c = \ln\left[\sqrt{m^2 \cdot \prod_{j=1}^{m-1} \operatorname{cosec}^2\left(\frac{\pi j}{m}\right)}\right]
	\]
	is a constant.
\end{theorem}
	
    \begin{remark}
    	We can also do a finite decomposition for the spacelike Scherk's tower in the Lorentz-Minkowski 3-space. The argument will follow in exactly the same way. We will avoid presenting it here to avoid repetition.
    \end{remark}
    \noindent
    Once again let $m$ be a positive integer and $n = m k_n + j$, $k_n \in {\mathbb Z}^+ \cup \{0\}$, $j =0,\dots,m-1$. Suppose that \(f(z)\) is a holomorphic function having an infinite product representation \(f(z) = z^k e^{g(z)} \prod_{n=0}^{\infty} (1-\frac{z}{a_n}) e^{P_q(\frac{z}{a_n})}\), where $P_q$ is the polynomial in \ref{thm: finite decomposition} for a constant degree $q$. This occurs under mild conditions on $\{a_n\}$ (see \cite{PonnusamySilverman2006}, Page 426, Equation 12.8). Then we have the following rearrangement of the infinite product:
    \[
    f(z)=z^ke^{g(z)} \prod_{n=0}^{\infty} \left(1-\frac{z}{a_n}\right) e^{P_q(\frac{z}{a_n})} = z^k e^{g(z)} \prod_{j=0}^{m-1} \left[ \prod_{k_n=0}^{\infty} \left(1- \frac{z}{a_{k_n m +j}}\right) \cdot e^{P_q\left(\frac{z}{a_{k_n m +j}}\right)}\right].
    \]
    If we call \(\mathcal{F}\left(\{a_n\}_{n=0}^{\infty}, q\right)= \prod_{n=0}^{\infty} \left(1-\frac{z}{a_n}\right) e^{P_q(\frac{z}{a_n})}\), then we have the following finite decomposition of \(f(z)\):
    \begin{equation}\label{eqn: FD equation}
    	f(z)=z^ke^{g(z)}\mathcal{F}\left(\{a_n\}_{n=0}^{\infty}, q\right)=z^ke^{g(z)}\prod_{j=0}^{m-1}\mathcal{F}\left(\{a_{mk_n+j}\}_{k_n=0}^{\infty}, q\right).
    \end{equation}
    With this we obtain the following theorem:
    
    \begin{theorem}\label{thm: weierstrass prod decomposition 1 FD}
    	Let $\{a_{n}\}_{n \geq 1}$ be a sequence of nonzero complex numbers, and let $f(z)$ be an entire function with zeros at $\{a_{n}\}$ listed according to their multiplicities. Suppose $f$ has a zero of order $k \geq 0$ at the origin and $f(z) = z^k e^{g(z)} \prod_{n=0}^{\infty} (1-\frac{z}{a_n}) e^{P_q(\frac{z}{a_n})}$, where $P_q$ is the polynomial in \ref{thm: finite decomposition} for a constant degree $q$. Set \(\mathcal{F}\left(\{a_n\}_{n=0}^{\infty}, q\right)= \prod_{n=0}^{\infty} \left(1-\frac{z}{a_n}\right) e^{P_q(\frac{z}{a_n})}\), a function with parameters \(\{a_n\}_{n=0}^{\infty}\) and \(q\). If $X$ is a graphical isotropic minimal surface with height function $h(x,y) = \operatorname{Re}(\ln f(z))$, then $h$ admits a decomposition on a simply connected domain excluding $\{a_{n}\}_{n \geq 1}$ and $0$, into the following components:
    	\begin{enumerate}
    		\item A finite sum representing $k$ logarithmoids of revolution, given by $k \ln|z|$;
    		\item A finite sum of isotropic minimal surfaces given by \(\operatorname{Re}\ln\left[\mathcal{F}\left(\{a_{mk_n+j}\}_{k_n=0}^{\infty}, q\right)\right]\) which is parametrised by the parameters \(\{a_{mk_n+j}\}_{k_n=0}^{\infty}\) and \(q\). They are similar to the isotropic minimal surface given by \(\operatorname{Re}\ln\left[\mathcal{F}\left(\{a_{n}\}_{n=0}^{\infty}, q\right)\right]\), with the only difference being that the zeros \(\{a_{n}\}_{n=0}^{\infty}\) are changed to \(\{a_{mk_n+j}\}_{k_n=0}^{\infty}\);
    		\item A graphical isotropic minimal surface given by the harmonic function defined by $\operatorname{Re}(g(z))$, where $g(z)$ is the entire function appearing in the Weierstrass product formula for \(f\).
    	\end{enumerate}
    	Moreover, since the Weierstrass product representation of an entire function is not unique, the corresponding decomposition of the isotropic minimal surface is likewise not uniquely determined.
    \end{theorem}
    
    \begin{proof}
    	Proof follows from equation \eqref{eqn: FD equation}.
    \end{proof}
    
    \begin{remark}
    	One can also consider the case where $X$ is a graphical isotropic minimal surface with height function $h(x,y) = \operatorname{Im}(\ln f(z))$; then $h$ also admits a finite decomposition. This can be easily done using the same argument. We will not give those calculation to avoid repetition.
    \end{remark}
    \begin{remark}
    	Finite decomposition for $c$-families can also be done using Theorem \ref{thm: finite decomposition}. However, the components in the finite decomposition need not be in a recognizable form.
    \end{remark}        
    
            \section{Applications to Lamellar Structures}

Twist grain boundary (TGB) phases in smectic-A liquid crystals consist of ordered arrays of screw dislocations, each locally modelled by a helicoid. Such structures have been extensively studied in the liquid crystal literature and are known to possess deep connections with classical minimal surfaces \cite{Kamien2001,Santangelo2006,Santangelo2007}.\\
\noindent
The decomposition formulas obtained in Section~\ref{Decomposition in the Isotropic Space} provide a natural geometric interpretation of these phases. In particular, the isotropic minimal surface associated with the Weierstrass data
\[
(F,G)=\left(1,\frac12\cot z\right)
\]
has a height function that can be expressed as an infinite superposition of helicoidal contributions. This representation coincides with the classical description of a single twist grain boundary as a periodic array of screw dislocations. Consequently, our decomposition recovers the standard TGB geometry from the perspective of isotropic minimal surface theory. Furthermore, a single grain boundary is known to be topologically equivalent to Scherk's first minimal surface, highlighting a classical connection between TGB phases and minimal surface geometry \cite{Kamien2001}. \\
\noindent
Furthermore, the significance of finite and infinite superpositions of harmonic graphs extends beyond liquid-crystal models. Such constructions have also appeared in the study of biological membrane morphologies, including the endoplasmic reticulum and plant thylakoids, where arrays of helicoidal defects give rise to lamellar structures(see \cite{daSilvaEfrati2021,Vasu2026}).

		\section{Appendix}

        \subsection{Weierstrass Product: Infinite and Finite}\label{sec: weierstrass factorisation}

        We begin by stating a version of Weierstrass Product Theorem (see \cite{PonnusamySilverman2006}, p. 428):
	
	\begin{theorem}\label{thm: weierstrass factorisation}
		Let $\{a_n\}_{n \geq 1}$ be a sequence of nonzero complex numbers, and let $f(z)$ be an entire function with zeros at $a_n$, listed according to their multiplicities. Suppose that $f$ has a zero of order $k \geq 0$ at the origin. Then there exists an entire function $g(z)$ such that:
		\[
		f(z) = z^k e^{g(z)} \prod_{n=1}^{\infty} E_n\left( \frac{z}{a_n} \right),
		\]
		where $E_n(\frac{z}{a_n}) = (1-\tfrac{z}{a_n})e^{P_n(\frac{z}{a_n})}$ is the Weierstrass primary factor.
	\end{theorem}
	\noindent
    We have the following theorem recalling a known result about finite decomposition.
	\begin{theorem}
    \label{thm: finite decomposition}{\bf (Finite Decomposition):} 		
		Let $m,n$ be  positive integers with $m < n$. Let us fix $m$. Let $n = k_n m + r_n$, where $r_n \in \{ 0,1,2,...,m-1\}$ and $k_n = 1,..., \infty$. 
		Let $\{a_n\}_{n \geq 1}$ be a sequence of nonzero complex numbers and let $f(z)$ be an entire function with zeros at $a_n$, listed according to their multiplicities. Suppose that $f$ has a zero of order $k \geq 0$ at the origin. Then there exists an entire function $g(z)$ such that:
		
		$$ f(z) =  z^{k} e^{g(z)}\prod_{n=1}^{m-1} (1-\frac{z}{a_n} ) e^{ P_n(\frac{z}{a_n})} \prod_{r_n=0}^{m-1} G_{r_n} (\frac{z}{a_n}) $$ 
		where,    $G_{r_n} (z)=  \prod_{k_n=1}^{\infty} (1 - \frac{z}{a_n} )e^{P_n(\frac{z}{a_n})}$  and $P_n(w) = w + \frac{w^2}{2} + ...+ \frac{w^n}{n}$  is convergent.
		\end{theorem}
		
		\begin{proof}
			Let $m, n$ be as in the statement. Since $m <n$, we can write $n = k_n m + r_n$ where $r_n \in \{ 0,1,2,...,m-1\}$ and $k_n = 1,..., \infty$. Let for a fixed $r_n$, $b_{k_n} = a_n$ where $r_n$ is fixed. $b_{k_n} $ is well defined, because fixing $r_n$, there is a one $n$ for each $k_n$. 

			Then 
			\begin{eqnarray*}
				z^{-k} e^{-g(z)} f(z) &=& \prod_{n=1}^{m-1} (1-\frac{z}{a_n} ) e^{ P_n(\frac{z}{a_n} )} \prod_{n=m}^{\infty} (1 - \frac{z}{a_n}) e^{P_n(\frac{z}{a_n})}  \\
				&=& \prod_{n=1}^{m-1} (1-\frac{z}{a_n} ) e^{ P_n(\frac{z}{a_n})}  \prod_{r_n=0}^{m-1}  \prod_{k_{n} =1}^{\infty}   (1 - \frac{z}{a_n}) e^{P_n( \frac{z}{a_n})}  \\
				&=& \prod_{n=1}^{m-1} (1-\frac{z}{a_n} ) e^{ P_n(\frac{z}{a_n})}  \prod_{r_n=0}^{m-1}  \prod_{k_{n} =1}^{\infty}   (1 - \frac{z}{a_n}) e^{P_n( \frac{z}{a_n})}  \\
				&=& \prod_{n=1}^{m-1} (1-\frac{z}{a_n} ) e^{ P_n(\frac{z}{a_n})} \prod_{r_n=0}^{m-1} G_{r_n} (\frac{z}{a_n}) 
			\end{eqnarray*}
			where for each fixed $r_n$,  $G_{r_n}(z) = \prod_{k_{n} =1}^{\infty}   (1 - \frac{z}{b_{k_n}}) e^{P_n( \frac{z}{b_{k_n}})}$. In this product formula for $G_{r_n}$, $P_n$ has degree $n$ where $n =k_n m + r_n$. 
			$H(w) =  \prod_{k_{n}=1}^{\infty}   (1 - \frac{w}{b_{k_n}}) e^{P_{k_n}( \frac{w}{b_{k_n}})}$ converges (by Weierstrass product formula). Since this product converges the product for $G_{r_n}$ also converges because we are multiplying by $H(w)$ by $e^{\Big{(}\frac{(w/b_{k_n})^{k_n+1}}{k_n+1} + \frac{(w/b_{k_n})^{k_{n}+2}}{k_{n}+2}+...+ \frac{(w/b_{k_n})^n}{n }\Big{)}}$. (Higher degree of the polynomial doesnot affect convergence, see for instance \cite{Conway}).  
		\end{proof}
\noindent
        Next we provide a sufficient condition on the zeros $\{a_n\}$ of $f$ that guarantees the existence of a  decomposition of the form $f(z)=\prod_{n=1}^{\infty} (1-\frac{z}{a_n} )$. The primary tool for this analysis is the Weierstrass $M$-test for infinite products. For completeness, we state a version of the $M$-test used here (see \cite[p.~417]{PonnusamySilverman2006}):
		\begin{theorem}\label{thm: M-test}
			Suppose that $\{f_n(z)\}$ is a sequence of functions such that $|f_n(z)| \leq M_n$ for all $z$ in a region $\Omega$. If $\sum_{n=1}^{\infty} M_n$ converges, then $\prod_{n=1}^{\infty} [1 + f_n(z)]$ converges uniformly in $\Omega$. In addition, if $f(z) = \prod_{n=1}^{\infty} [1 + f_n(z)]$ and each $f_n(z)$ is analytic in $\Omega$, then $f(z)$ is analytic in $\Omega$. 
		\end{theorem}

\subsection{Wallis product formula}\label{sec: wallis product}

Let $i\in \{0,1,\dots,m-1\}$.
The $i^{th}$-Wallis product is given by the following theorem:

\begin{theorem}\label{thm: Wallis product}
	$\prod_{ n \in \mathbb{Z} \setminus \{0\}} \left[ \frac{\pi^2 n^2 }{\pi^2 (n + i/m)^2}\right] = \left[\frac{\pi i }{m \sin(\pi i/m )} \right]^2$.
\end{theorem}

\begin{proof}
	We start by expanding the infinite product formula on the LHS:
	\small
	\begin{equation*}
	\prod_{n \neq 0} \left[ \frac{\pi^2 n^2 }{\pi^2 (n + i/m)^2}\right]=\prod_{n=1}^{\infty} \left[ \frac{n}{(n + i/m)}\right]^2\left[ \frac{-n }{(-n + i/m)}\right]^2=\prod_{n=1}^{\infty} \left[ \frac{1}{1 - \tfrac{i^2}{n^2m^2}}\right]^2=\left[ \frac{1}{\prod_{n=1}^{\infty}\left[1 - \frac{\pi^2 i^2/m^2}{n^2\pi^2}\right]}\right]^2.
	\end{equation*}
	\normalsize
	Now we observe that the denominator is the infinite product formula for \(\tfrac{\sin(i\pi/m)}{i \pi/m}\). Substituting this fact yields the required \(i^{\text{th}}\) Wallis product formula.
\end{proof}

			\section{Acknowledgement}

			This research was supported in part by the International Centre for Theoretical Sciences (ICTS) for the Geometry and Analysis of Minimal Surfaces Program (code: ICTS/GAMS2025/08). Sam K. Mathew and Rukmini Dey acknowledge support from the Department of Atomic Energy, Government of India, under project no. RTI4001. Rahul Kumar Singh is partially supported by the MATRICS grant (File No. MTR/2023/000990), which has been sanctioned by the SERB/ANRF.
			
			\bibliography{Bibliography}
            \bibliographystyle{plain}

		\end{document}